\documentclass[EJP]{ejpecp} 
\pdfoutput=1

\usepackage{enumerate,color}

\usepackage{epsfig}
\usepackage[mathscr]{eucal}

\SHORTTITLE{Predator-prey dynamics on infinite trees} 

\TITLE{Extinction probability and total progeny of predator-prey dynamics on infinite trees} 

\AUTHORS{%
  Charles~Bordenave\footnote{CNRS \& Universit\'e Toulouse III, France.
    \EMAIL{charles.bordenave@math.univ-toulouse.fr}}}

\KEYWORDS{SIR models ; predator-prey dynamics ; branching processes} 

\AMSSUBJ{60J80} 

\SUBMITTED{October 10, 2012} 
\ACCEPTED{January 10, 2014} 

\ARXIVID{1210.2883} 
\ARXIVPASSWORD{pvkyj} 

\VOLUME{19}
\YEAR{2014}
\PAPERNUM{20}
\DOI{v19-2361}

\ABSTRACT{We consider the spreading dynamics of two nested invasion clusters on an infinite tree. This model was defined as the chase-escape model by Kordzakhia and it admits a limit process, the birth-and-assassination process, previously  introduced by Aldous and Krebs. On both models, we prove an asymptotic equivalent of the extinction probability near criticality. In the subcritical regime, we give a tail bound on the total progeny of the preys before extinction.}

\newcommand{\dE}{\mathbb{E}}

\newcommand{\dN}{\mathbb{N}}
\newcommand{\dP}{\mathbb{P}}
\newcommand{\dR}{\mathbb{R}}

\newcommand{\dZ}{\mathbb{Z}}

\newcommand{\cB}{\mathcal{B}}
\newcommand{\cD}{\mathcal{D}}

\newcommand{\cH}{\mathcal{H}}

\newcommand{\cS}{\mathcal{S}}

\newcommand{\cX}{\mathcal{X}}

\newcommand{\si}{\sigma}

\newcommand{\veps}{\varepsilon}

\newcommand{\BRA}[1]{{{\left\{#1\right\}}}} 
\newcommand{\PAR}[1]{{{\left(#1\right)}}} 
\newcommand{\SBRA}[1]{{{\left[#1\right]}}} 

\newcommand{\EXP}{\mathrm{Exp}}
\newcommand{\IND}{\mathbf{1}}

\newcommand{\so}{\text{\o}} 

\renewcommand{\qed}{\hfill{\ \ \rule{2mm}{2mm}} \vspace{0.2in}}

\newcommand{\ind}{1\hspace{-2.3mm}{1}}
\def\og{{\overline \gamma}}

\begin{document}

\section{Introduction}

The {\em chase-escape process} is a stochastic predator-prey dynamics which was studied by Kordzakhia \cite{kordzakhia} on a regular tree. In an earlier paper, Aldous and Krebs \cite{aldouskrebs} had introduced the {\em birth-and-assassination (BA) process}. The latter model can be seen as a natural limit of the chase-escape model. In \cite{MR2453554} the two models were merged into the {\em rumor scotching process}.  The original motivation of Aldous and Krebs was then to
analyze a scaling limit of a queueing process with blocking which
appeared in database processing, see Tsitsiklis, Papadimitriou and
Humblet \cite{tsitsiklis}. As pointed in \cite{MR2453554},  the
BA process is also the scaling limit of a rumor spreading model which is motivated by network epidemics and
dynamic data dissemination (see for example,  \cite{surveySIR},
\cite{andersson}, \cite{dynamicinformation}).

We may conveniently define the chase-escape processes as a SIR dynamics (see for example \cite{surveySIR} or \cite{andersson} for some background on standard SIR dynamics). This process represents the
dynamics of a rumor/epidemic spreading on the vertices of a graph
along its edges. A vertex  may be unaware of the rumor/susceptible
(S), aware of the rumor and spreading it as true/infected (I), or
aware of the rumor and trying to scotch it/recovered (R).

We fix a locally finite connected graph $G = (V,E)$. The chase-escape process is described by
a Markov process on $\cX=    \{ S, I , R\} ^V$. If $\{u,v\} \in E$, we write $u \sim v$. For $v \in V$, we also define the $\cX \to \cX$ maps
$I_v$ and $R_v$   by : for $x = ( x_u)_{u \in V}$, $( I_{v}(x) )_u = (R_{v}(x) )_u  = x_u$, if $u \ne v$ and $(I_{v} (x))_v =  I $,  $(R_{v}(x))_v = R$. Let $\lambda \in (0,1)$ be a fixed infection intensity. We then define the Markov process with
transition rates:
\begin{eqnarray*}
K (x,  I_{v}(x)) & =  & \lambda   \IND ( x_v  = S )  \sum_{ u  {\sim} v   } \IND ( x_u =  I ), \\
 K(x, R_{v}(x)) & =  & \IND ( x_v =  I ) \sum_{ u {\sim} v   } \IND ( x_u = R),
\end{eqnarray*}
and all other transitions have rate $0$. In words, a susceptible vertex is infected at rate $\lambda$ by its infected neighbors, and  an infected vertex is recovered at rate $1$ by its recovered neighbors. The absorbing states of this process are the states without $I$-vertices or with only $I$ vertices. In this paper, we are interested by the behavior of the process when at time $0$  there is a non-empty finite set of $I$ and $R$-vertices.

In \cite{kordzakhia}, this model was described as a predator-prey
dynamics: each vertex may be empty (S), occupied by a prey
(I) or occupied by a predator (R). The preys spread on unoccupied
vertices and predators spread on vertices occupied by preys.  If
$G$ is the $\dZ^d$-lattice and if there is no $R$-vertex, the
process is the original Richardson's model  \cite{richardson}.
With $R$-vertices, this process is a variant of the two-species
Richardson model with prey and predators, see for example
H{\"a}ggstr{\"o}m and Pemantle \cite{haggstrom98}, Kordzakhia  and
Lalley \cite{kordzakhia05}. There is a growing cluster of (I)-vertices spreading over (S)-vertices and a nested growing cluster of (R)-vertices spreading on (I)-vertices.

The chase-escape process differs from the classical SIR dynamics on the transition from $I$ to $R$: in the classical SIR dynamics, a $I$-vertex is recovered at rate $1$ independently of its neighborhood.

\paragraph{Chase-escape process on a tree}

If the graph $G = T = (V,E)$ is a rooted tree, the process is much simpler to study. We denote by $\so$ the root of $T$. For the range of initial conditions of interest (non-empty finite set of $I$ and $R$-vertices),  there is no real loss of generality to study the chase-escape process on the tree $T^\downarrow$ obtained from $T$ by adding a particular vertex, say $o$, connected to the root of the tree. At time $0$, vertex $o$ is in state $R$, the root $\so$ is in state $I$, while all other vertices are in state $S$ (see figure \ref{fig:ic}). We shall denote by $X(t) \in \{ S, I , R \}^{V}$ our Markov process  on the tree $T^\downarrow$. Under $\dP_{\lambda}$, $X$ is the chase escape process on $T^\downarrow$ with infection rate $\lambda$.

\begin{figure}[htb]
\centering \scalebox{0.6}{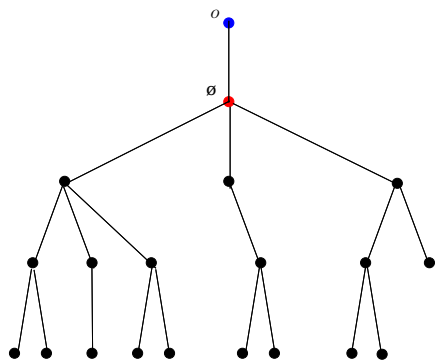}
\caption{The initial condition : the root is $I$, $o$ is $R$, all other vertices are $S$.}\label{fig:ic}
\end{figure}

We say that the Markov process $X$ {\em gets extinct} if at some (random) time $\tau < \infty,$ there is no $I$-particle. Otherwise the process is said to {\em survive}. We define the probability of extinction as
$$
q_T(\lambda) = \dP_{\lambda} ( X \hbox{ gets extinct}).
$$
Obviously, if $T$ is finite then $q_{T} (\lambda) =1$ for any $\lambda \geq 0$. Before stating our results, we first need to introduce some extra terminology. 

There is a canonical way to represent the vertex set $V$ as a subset of  $\dN^f = \cup_{k=0} ^{\infty} \dN^k$ with $\dN^0 = \so$ and $\dN = \{1 , 2 \cdots \}$. If $k \geq 1$ and $v\in V $ is at distance $k$ from the root, then $v   = (i_1, \cdots, i_k) \in V \cap \dN^k$. The {\em genitor} of $v$ is $(i_1, \cdots, i_{k-1})$: it is the first vertex on the path from $v$ to the root $\so$ of length $k$. The {\em offsprings} of $v$ are set of vertices who have genitor $v$. They are indexed by $(i_1, \cdots, i_{k} , 1), \cdots , (i_1, \cdots, i_{k} , n_v)$, where $n_v$ is the number of offsprings of $v$. The {\em ancestors} of $v$ is the set of vertices $(i_1,\cdots, i_\ell)$, $0 \leq \ell \leq k-1$ with the convention $i_0 = \o$. Similarly, the \emph{$n$-th generation offsprings} of $v$ are the vertices in $V \cap \dN^{k+n}$ of the form $(v,i_{k+1}, \cdots, i_{k+n})$.

Recall that the \emph{upper growth rate} $d \in [1,\infty]$ of a rooted infinite tree $T$ is defined as
$$
d = \limsup_{k \to \infty} |V_k|^{1/k},
$$
where $V_k = V \cap \dN^k$ is the set of vertices at distance $k$ from the root $\o$ and $| \cdot |$ denotes the cardinal of a finite set. The {\em lower growth rate} is defined similarly with a $\liminf$. When the $\liminf$ and the $\limsup$ coincide, this defines the \emph{growth rate} of the tree. 

For example, for integer $d \geq 1$, we define the \emph{$d$-ary tree} as the tree where all vertices have exactly $d$ offsprings\footnote{It would be more proper to call this tree the complete infinite $d$-ary tree.}. Obviously, the $d$-ary tree has growth  rate $d$. More generally, consider a realization $T$ of a Galton-Watson tree with mean number of offsprings $d \in (1,\infty)$. Then, the Seneta-Heyde Theorem \cite{MR0234530,MR0254929} implies that, conditioned on $T$ infinite, the growth rate of $T$ is a.s. equal to $d$. For background on random trees and branching processes, we refer to  \cite{MR2047480,lyonsbook}.

For integer $n \geq 1$, we define $T^{*n}$ as the rooted tree on $V$ obtained from $T$ by putting an edge between all vertices and their $n$-th generation offsprings. For real $d > 1$, we say that $T$ is a \emph{lower $d$-ary} if for any $1 < \delta < d$, there exist an integer $n \geq 1$ and $v \in V$ such that the subtree of the descendants of $v$ in $T^{*n}$ contains a $\lceil {\delta}^n \rceil$-ary tree. Note that if $T$ is lower $d$-ary then its lower growth rate is at least $d$. Also, if $T$ is the realization of a Galton-Watson tree with mean number of offsprings $d \in (1,\infty)$ then, conditioned on $T$ infinite, $T$ is a.s. lower $d$-ary (for a proof see Lemma \ref{le:daryGWT} in appendix).

The first result is an extension of \cite[Theorem 1]{kordzakhia} where it is proved for \(d\)-ary trees. It describes the phase transition of the event of survival.

\begin{theorem}\label{th:RSstab}
Let $d > 1$ and \[\lambda_1   =    2d -1  - 2 \sqrt{ d (d-1)} . \]
If \(0 < \lambda <  \lambda_1\) and the upper growth rate of $T$ is at most $d$, then $q_T(\lambda) = 1$. If \(\lambda >  \lambda_1 \) and $T$ is lower $d$-ary, then $0 < q_T (\lambda) < 1$.
\end{theorem}

Note that in the classical SIR dynamics, it is easy to check that the critical value of $\lambda$ is $\lambda = 1  / (d-1)$. Also, for any $d > 1$, $\lambda_1 < 1$ and,
\begin{equation}\label{eq:asymlambda}
\lambda_1 \sim_{ d \uparrow \infty} \frac{1}{4d}.
\end{equation}

The proof of Theorem \ref{th:RSstab} will follow a strategy parallel to \cite{kordzakhia,aldouskrebs}. We employ techniques akin to the study the infection process in the Richardson
model. They will be based on large deviation estimates on the probability that a single vertex is $I$ at time $t$.

 To our knowledge, there is no known closed form expression for
the extinction probability $q_T (\lambda)$. Our next result determines an asymptotic equivalent for the probability of survival  for \(\lambda\) close to $\lambda_1$.
Our method does not seem to work on the sole assumption that $T$ has growth rate $d > 1$ and is lower $d$-ary. We shall assume that $T$ is a realization of a Galton-Watson tree with offspring distribution $P$ and $$d = \sum_{k=1} ^ \infty k P (k)  > 1.$$
We consider the annealed probability of extinction:
$$
q ( \lambda ) = \dE' [  q_T ( \lambda ) ]  = \dP'_\lambda( X \hbox{ gets extinct}) ,
$$
where the expectation $\dE'(\cdot)$ is with respect to the randomness of the tree and $\dP'_\lambda ( \cdot )  = \dE'  \left( \dP_\lambda ( \cdot ) \right) $ is the probability measure with respect to the joint randomness of $T$ and $X$. Note that in the specific case $d$ integer and $P (d) = 1$, $T$ is the $d$-ary tree and the measures $\dP'_\lambda$ and $\dP_\lambda$ coincide.

\begin{theorem} \label{th:RSext}
Assume further that the offspring distribution has finite second moment. There exist  constants $c_0, c_1 > 0$ such that for all $\lambda_1  < \lambda < 1 $,
$$
c_0 \omega^3  e^{  - \frac{  ( 1 -\lambda_1 ) \pi }{ 2 ( d ( d-1) ) ^{ 1/ 4 } }  \omega^{-1} }   \leq   1- q(\lambda) \leq  c_1  e^{  - \frac{  ( 1 -\lambda_1 ) \pi }{ 2 ( d ( d-1) ) ^{ 1/ 4 } }  \omega^{-1} },
$$
with
$$
\omega = \sqrt{\lambda - \lambda_1 }.
$$
\end{theorem}

Note that the behavior depicted in Theorem \ref{th:RSext} contrasts with the classical SIR dynamics, where $1 - q(\lambda)$ is of order $( \lambda  (d - 1)  -1 )_+$.  This result should however be compared to similar results in the Brunet-Derrida model of branching random walk killed
 below a linear barrier, see Gantert, Hu and Shi \cite{MR2779399} and also B\'erard and Gou\'er\'e \cite{MR2774095}.
 As in this last reference, our approach is purely analytic. We will first check that $q(\lambda)$ is
 related to  a second order non-linear differential equation. Then, we will rely on comparisons with linear
 differential equations. A similar technique was already used by Brunet and Derrida \cite{MR1473413}, and notably
 also in Mueller, Mytnik and Quastel \cite[section 2]{MR2793860}. 

A possible parallel with the Brunet-Derrida model of branching random walk killed
 below a linear barrier is the following. Consider a branching random walk on $\dZ$ started from a single particle at site $0$ where the particles may only move by one step on the right. If we are only concerned by the extinction, we can think of this process as some branching process without walks where a particle at site $k$ gives birth to particles at site $k+1$. We can in turn represent this process by a growing random tree where the set of vertices at depth $k$ is the set of particles at site $k$. Hence (I)-vertices play the role of the particles, the branching mechanism is the spreading of the (I)-vertices over the (S)-vertices and the set of (R)-vertices is a randomly growing barrier which absorbs the particles/(I)-vertices.  Kortchemski \cite{IK} has recently built an explicit coupling of a branching random walk with the chase-escape process on a tree.

In the case $ 0 < \lambda < \lambda_1$, the process $X$ stops a.s.\ evolving after some finite $\tau$. We define $Z $ as the total progeny of the root, i.e. the total number of
recovered vertices (excluding the vertex $o$ of $T^\downarrow$) at time $\tau$. It is the number of vertices which will have been infected before the process reaches its absorbing state. We define the annealed parameter:
$$
\gamma (\lambda) = \sup \left\{ u  \geq 0 :  \dE'_{\lambda} [ Z ^u] < \infty \right \}.
$$
The scalar $\gamma(\lambda)$ can be though as a power-tail exponent of the variable $Z$ under the annealed measure $\dP'_\lambda$. In particular, for any $ 0 < \gamma < \gamma(\lambda)$, from Markov Inequality, there exists a constant $c > 0$ such that for all $t \geq 1$,
$ \dP'_{\lambda} ( Z \geq  t ) \leq c t^{-\gamma}.$ Conversely, if there exist $c,\gamma >0$ such that for all $t \geq 1$,
$ \dP'_{\lambda} ( Z \geq t ) \leq c t^{-\gamma}$, then $\gamma(\lambda) \geq \gamma$.
We define
$$
\gamma_{P} = \sup \left\{ u  \geq 1 :  \sum_{k=1} ^ \infty k^u P(k) < \infty \right \} \geq 1.
$$
\begin{theorem} \label{th:RStail}
\begin{itemize}
\item[(i)]  For any $0 < \lambda < \lambda_1$,
$$
\gamma (\lambda) =\min\left(\frac{ \lambda^2 - 2 d
\lambda +1 - (1 - \lambda) \sqrt{ \lambda^2 - 2 \lambda ( 2 d - 1) + 1}}{2\lambda (d -1)} , \gamma_P \right).$$
\item[(ii)] Let $1 \leq u < \gamma_{P}$,  $A_u = u^2 (d-1) + 2 u d  + ( d
-1)$, and 
$$
 \lambda_u = \frac{A_u - \sqrt{ A_u ^2 - 4 u^2}}{2 u}.
$$
If $\lambda < \lambda_u $ then $\dE'_\lambda[  Z^u ] $ is finite. If  $\lambda >
\lambda_u$, $ \dE'_\lambda[  Z ^u ]$ is infinite. \end{itemize}
\end{theorem}
It is straightforward to check that $(i)$ is equivalent to $(ii)$. Also, for $u=1$,  $\lambda_u$ coincides with $\lambda_1$ defined in Theorem \ref{th:RSstab}. It follows that $\gamma(\lambda) \geq 1$ for all $0 < \lambda < \lambda_1$. Theorem \ref{th:RStail}  contrasts with classical SIR dynamics. For example, if $T$ is the $d$-ary tree, for all $\lambda < 1/ (d -1)$  there exists a constant
$c>0$ such that $\dE'_\lambda \exp(c S) < \infty$ where $S$ is the total progeny in the classical SIR dynamics. Here, the heavy-tail
phenomenon is an interesting feature of the chase-escape process. Intuition suggest that large values of $Z$ come from a (I)-vertex which is not recovered before an exceptionally long time. Indeed, in the chase escape process, a (I)-vertex which is not recovered by time $t$ will typically have a progeny which is exponentially large in $t$ (this is not the case in the classical sub-critical SIR dynamics, the progeny of such vertex will typically be of order $1$) .  A similar phenomenon appears also in the Brunet-Derrida model, see Addario-Berry and Broutin \cite{MR2834719}, A{\"{\i}}d{\'e}kon \cite{MR2737710} and A{\"{\i}}d{\'e}kon Hu and Zindy \cite{AHZ}. Note finally that
$$\gamma (\lambda) \sim_{\lambda \downarrow 0}
\min \PAR{ \frac{1}{(d-1)\lambda} , \gamma_P } \quad \hbox{ and } \quad  \gamma (\lambda) \sim_{\lambda \uparrow \lambda_1} 1.$$

By recursion, we will also compute the moments of $Z$. The computation of the first moment gives
\begin{theorem} \label{th:RStail2} 
If $ 0 < \lambda \leq \lambda_1$ and $ \Delta = \lambda^2 - 2 \lambda ( 2 d - 1) + 1$, then
$$ \dE'_\lambda [ Z ] =  \frac {2d}{(d-1)(1 + \lambda + \sqrt \Delta)}- \frac{1}{d-1}. $$
\end{theorem}

Theorem \ref{th:RStail2} implies a surprising right discontinuity of the function $\lambda \mapsto \dE'_\lambda Z $ at the critical intensity $ \lambda = \lambda_1$:
$ \dE'_{\lambda_1} Z  = 2d / ( ( d-1) (1 + \lambda_1)) - 1  / ( d-1)  <  \infty$. Again, this discontinuity contrasts with what happens
in a standard Galton-Watson process near criticality, where for
$0< \lambda < 1/(d-1)$, $\dE'_\lambda Z$ is of order  $(1  -  (d-1) \lambda  ) ^{-1}$. From Theorem \ref{th:RStail2}, we may fill the gap in Theorem \ref{th:RSstab} in the specific case of a realization of a Galton-Watson tree.
\begin{corollary}
\label{cor:N}
Let $T$ be a Galton-Watson tree with mean number of offsprings $ d$. Then a.s. $q_T (\lambda_1) =1$.
\end{corollary}

The method of proofs of Theorems \ref{th:RStail}-\ref{th:RStail2} will be parallel to arguments in \cite{MR2453554} on the birth-and-assassination process.

\paragraph{The birth-and-assassination process}

We now turn to the BA process. It is a scaling limit in $ d \to \infty$ of the chase-escape  process  on the $d$-ary tree when $\lambda$ is rescaled in $\lambda/ d$.

Informally, the process can be described as follows. We start from a root vertex that produces offsprings according to a Poisson process of rate
\(\lambda.\) Each offspring in turn produces children according to independent Poisson processes and
so on. The children of the root are said to belong to the first generation and
their children to the second generation and so forth. Independently, the root vertex is \emph{at risk} at time $0$ and dies after a
random time \(D_\so\) that is exponentially distributed with mean $1$. Its offsprings become at risk
after time \(D_\so\) and the process continues in the next generations. We now make precise the above description.

As above, $\dN^f = \cup_{k=0} ^{\infty} \dN^k$ denotes the set of finite \(k-\)tuples
of positive integers (with $N^0 =\so$). Elements from this set are used to index
the offspring in the BA process. Let $\{\Xi_v\},
v \in \dN^f$, be a family of independent Poisson processes with
common arrival rate $\lambda$; these will be used to define the offsprings.
Let $\{D_v\}, v \in \dN^f$, be a
family of independent, identically distributed (iid) exponential random variables
with mean $1;$ we use them to assign the lifetime for the appropriate offspring.
The families $\{ \Xi_v \}$ and $\{ D_v \}$ are independent. The process starts at
time $0$ with only the root, indexed by $\so$.
This produces offspring at the arrival times determined by
$\Xi_\so$ that enter the system with indices $(1)$,
$(2)$, $\cdots$ according to their birth order. Each new vertex
$v,$  immediately begins producing offspring
determined by the arrival times of $\Xi_v.$ The offspring of $v$ are
indexed $(v, 1)$, $(v, 2)$, $\cdots$ also according to birth
order. The root is {\em at risk} at time $0$. It
continues to produce offspring until time $T_\so =
D_\so$, when it dies. Let $k>0$ and let
$v=(n_1,\cdots,n_{k-1},n_k)$, $v'=(n_1,...,n_{k-1})$ denote a vertex and its genitor. When a
particle $v'$ dies (at time $T_{v'}$), the particle $v$ then becomes at
risk; it in turn continues to produce offspring until time $T_v =
T_{v'} + D_v$, when it dies (see figure \ref{fig:ba2}).

The BA process can be equivalently described as a Markov process $X(t)$ on $\{ S, I , R\}^{\dN^f}$, where a particle/vertex in state $S$ is not yet born, a particle in state $I$ is alive and a particle in state $R$ is dead. A particle is at risk if it is in state $I$ and its genitor is in state $R$. We use the same notation as above : under $\dP_\lambda$, the process $X(t)$ has infection rate $\lambda > 0$, $q(\lambda)$ is the probability of extinction and so on.

\begin{figure}[htb]
\begin{center}
\includegraphics[angle=0,height = 6cm]{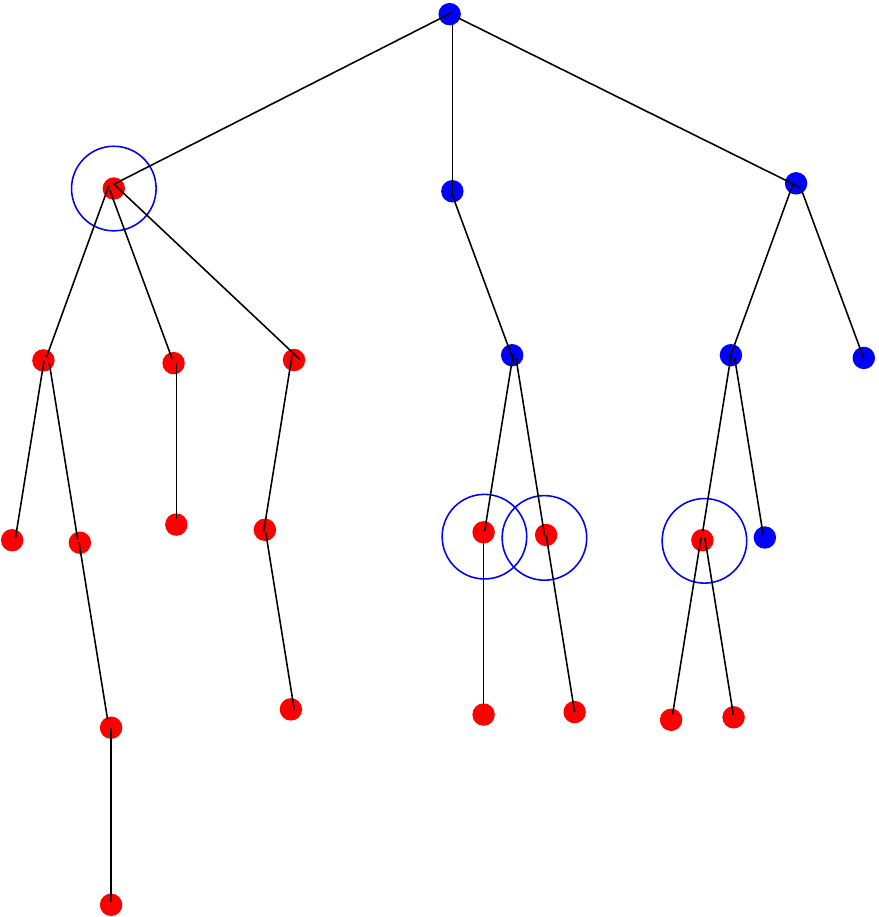}
\caption{Illustration of the birth-and-assassination process, living particles are in red, dead particles in blue, particles at risk are encircled.}\label{fig:ba2}
\end{center}
\end{figure}

The following result from \cite{aldouskrebs}  describes the phase transition on the probability of survival as a function
of \(\lambda\).
\begin{theorem}\label{aldous_thm}
Consider the BA process with rate $\lambda > 0$. If \(\lambda \in [0, 1/4],\) then  $q(\lambda ) = 1$, while if \(\lambda > \frac{1}{4},\)
$0< q(\lambda) < 1.$
\end{theorem}
The critical case $\lambda = 1/4$ was established in \cite{MR2453554}. Note also that the threshold $\lambda = 1/4$ is consistent with \eqref{eq:asymlambda}.

Our final result is the analog of Theorem \ref{th:RSext}.
\begin{theorem}\label{th:mainex}
Consider the BA process and assume that $\lambda > 1/4$.  There exist constants $c_0, c_1 > 0$ such that for all $1/4  < \lambda < 1 $,
$$
c_0 \omega^3  e^{  - \frac{ \pi }{2 }  \omega^{-1}  }   \leq  1- q(\lambda) \leq  c_1  \omega ^{-1}   e^{  - \frac{ \pi }{2 } \omega^{-1}  },
$$
with
$$
\omega = \sqrt{\lambda - \frac 1 4  }.
$$
\end{theorem}

Note that the analog of Theorems \ref{th:RStail}-\ref{th:RStail2} was already performed in \cite{MR2453554}. The remainder of the paper is organized as follows. In section \ref{sec:BA}, we start with the study of the BA process and prove Theorem \ref{th:mainex}. Proofs on the BA process are simpler and this section is independent of the rest of the paper. We then come back to the chase escape process: in section \ref{pf1}, we prove  Theorem~\ref{th:RSstab}, in section \ref{sec:RSext}, we prove Theorem \ref{th:RSext}. Finally, in section \ref{sec:RStail}, we prove Theorems \ref{th:RStail}-\ref{th:RStail2}.

\section{Proof of Theorem \ref{th:mainex}}
\label{sec:BA}

\subsection{Differential equation for the survival probability}

We first determine the differential equation associated to the probability of extinction for the BA process.
Define $Q_\lambda (t)$ to be the extinction probability given that the root dies at time $t  > 0$ so that
\begin{equation}\label{eq:qQ}
q(\lambda) = \int_0 ^\infty Q_{\lambda} (t) e^{-t} dt
\end{equation}
and $Q_\lambda (0) =1$. Let $\{\xi_i\}_{i \geq 1}$ be the arrival
times of $\Xi_{\so}$ with $0 \leq \xi_1 \leq \xi_2 \leq \cdots$. For
integer $i$ with $1 \leq \xi_i \leq D_{\so}$, we define $\cB_i$ as the subprocess on
particles $i \dN^f$ with ancestors $i$. For the process $\cB$ to get extinct, all
the processes $\cB_i$ must get extinct. Conditioned on $\Xi_{\so}$, and on the root to
die at time $D_{\so} = t$, the evolutions of the $(\cB_i)$ then become independent. Moreover, on this
conditioning, $\cB_i$ is a birth-and-assassination
process conditioned on their root to be at risk
at time $t - \xi_i$. Hence, we get
$$
Q_\lambda ( t ) = \dE_\lambda \SBRA{ \prod_{\xi_i \leq  t} Q_\lambda ( t - \xi_i + D_i )} = \dE_\lambda \SBRA{ \prod_{\xi_i \leq  t} Q_\lambda (  \xi_i + D_i )},
$$
where $\{\xi_i\}_{i \geq 1}$ is a Poisson point process of
intensity $\lambda$ and $(D_i), i \geq 1,$ independent exponential variables with parameter $1$. Using
L\'evy-Khinchin formula, we deduce
\begin{eqnarray*}
Q_\lambda(t)  &=&
\exp \PAR{ \lambda \int_0 ^t  (  \dE Q_\lambda ( x + D_1)   - 1 ) dx }  \nonumber\\
&=& \exp \PAR{ \lambda \int_0 ^t \int_0 ^\infty(  Q_\lambda ( x + s)  - 1 ) e^{-s} ds dx }.
\end{eqnarray*}
So finally, for any $t \geq 0$,
\begin{eqnarray*}
Q_\lambda(t) &  =  & \exp \PAR{- \lambda t  + \lambda \int_0 ^t e^{x} \int_x ^\infty Q_\lambda ( s) e^{-s} ds dx}.
\end{eqnarray*}
We perform the change of variable
\begin{equation}\label{eq:defx}
x (t)  =  - \ln Q_\lambda (t).
\end{equation}
We find that for any $t \geq 0$,
\begin{eqnarray}\label{eq:fpQlambda}
x(t) =  \int_0 ^t e^{x} \int_x ^\infty \varphi( x (s) )  e^{-s} ds dx,
\end{eqnarray}
where 
$$
\varphi(y) = \lambda(1-e^{-y}).
$$
Differentiating \eqref{eq:fpQlambda} once gives
\begin{equation}\label{eq:xprime}
x' (t) = e^{t}  \int_t ^\infty \varphi( x(s))  e^{-s} ds,
\end{equation}
Now, multiplying the above expression by $e^{-t}$ and differentiating once again, we find that \(x(t)\)
satisfies the differential equation
\begin{equation}\label{eq:odex}
x'' - x ' + \varphi(x) = 0.
\end{equation}
This non-linear ordinary differential equation is not a priori easy to solve. However, in the neighborhood of \(\lambda =1/4\) it is possible to obtain an asymptotic expansion as explained below. The idea will be to linearize the ODE near $(x(0),x'(0)) = (0,0)$ and look at the long time behavior of the solutions of the linearized ODE. The critical value $\lambda  = 1/4$ appears to be the threshold for oscillating solutions of the linearized ODE. From a priori knowledge on the long time behavior of the solution of \eqref{eq:odex} of interest (studied in \S\ref{subsec:afpeq}), we will obtain an asymptotic equivalent for $(x(0),x'(0))$ as $\lambda \downarrow 1/4$ (in \S \ref{subsec:prfmainex}). 


\subsection{A fixed point equation}
\label{subsec:afpeq}
Let $\cH$ be the set of measurable
functions $f : \dR_+ \to \dR_+$ such that $f(0) = 0$ and for any $a >0$,
$$\lim_{s \to \infty} e^ { - a s }f(s)  = 0.$$
We define the map $A : \cH \to \cH$ defined by
\begin{equation}\label{eq:defT}
A (y) (t) = \int_0 ^t e^{x} \int_x ^\infty \varphi( y (s) )  e^{-s} ds dx.
\end{equation}
Using $\| \varphi \|_\infty = \lambda < \infty$, it is straightforward to check
that $A(y)$ is indeed an element of $\cH$ ($A(y)(t)$ it is
bounded by $\|\varphi \|_\infty t$). Note also that $y \equiv 0$ is a solution
of the fixed point equation
$$
y = A (y).
$$
Consider the function  $x$ defined by \eqref{eq:defx}. Using \eqref{eq:fpQlambda} we find that $x \in \cH$ and satisfies also the fixed point $x = A (x)$. If $\lambda > 1/4$, we know from Theorem \ref{aldous_thm} that $x \not\equiv 0$.

In the sequel, we are going to study  any non-trivial fixed point of $A$. To this end, let $x \in \cH$ such that $x = A( x)$ and $x \not\equiv 0$. By induction, it
follows easily that $t \mapsto x(t)$ is twice continuously differentiable. In particular, since
$x(s) \geq 0$, $x'(t) \geq 0$ and the function $x :\dR_+ \to \dR_+$ is non-decreasing. Moreover,
by assumption there exists $t_0 > 0$ such that $x(t_0) > 0$. Since $x$ is non-decreasing, we deduce
that $x(t) > 0$ for all $t > t_0$. Then, using again \eqref{eq:xprime}, we find that for all $t \geq 0$,
\begin{equation}\label{eq:Uyprime}
0 < x'(t) <  \lambda.
\end{equation}
From the argument leading to \eqref{eq:odex}, $x$ satisfies \eqref{eq:odex} and we are looking for a specific non-negative solution of \eqref{eq:odex} which satisfies
$x(0) = 0$.
To characterize completely this solution, it would be enough to compute $x'(0)$ (which is necessary positive since $x(0) =x'(0) = 0$ corresponds to the trivial solution $x \equiv 0$). We first give some basic properties of the phase portrait, see figure \ref{fig:ppBA}.  We define $X(t) = (x (t) , x'(t))$ so that
\begin{equation}\label{eq:defFX}
X ' = F (X)
\end{equation}
with $F((x_1, x_2) ) = ( x_2, x_2- \varphi(x_1))$. We also introduce the set
$$
 \Delta = \{ (x_1, x_2) \in \dR_+^2 : \varphi (x_1) < x_2  <  \lambda \}.
$$
\begin{figure}
\centering \scalebox{0.6}{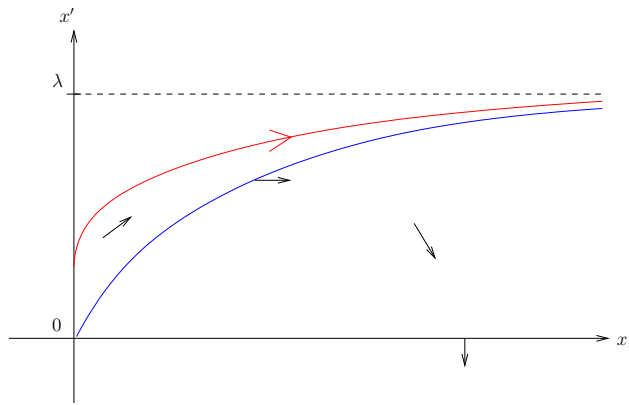}
\caption{Illustration of the phase portrait. In blue, the curve $L$, in red the curve $\Phi$ of Lemma \ref{le:trajODE}.}\label{fig:ppBA}
\end{figure}
\begin{lemma}\label{le:trajODE}
Let $x \in \cH$ such that  $x = A(x)$ and $x \not\equiv 0$. Then $x'(0) > 0$, $x$ satisfies \eqref{eq:odex} and for all $t \geq 0$,
\begin{equation}\label{eq:Xbox}
X ( t ) \in  \, \Delta.\end{equation}
Moreover $$\lim_{t \to \infty} x'(t) =  \lambda. $$
\end{lemma}
\begin{proof}
We have already checked that $x$ satisfies \eqref{eq:odex} and $x'(0) >0$. Let us now prove that \eqref{eq:Xbox} holds. Define the trajectory $\Phi = \{ X(t) \in \dR_+^2 : t \geq 0 \}$. Since for all $t \geq 0$, $X(t)'_1 = F(X(t))_1 > 0$, $\Phi$ is the graph of  a differentiable function $ f : [0,S) \to \dR_+$ with $f(0) = x'(0) > 0$:
$$
\Phi = \{ (s,f(s)) : s \in [0,S) \},
$$
with $S  = \lim_{ t \to \infty} x(t) \in (0,\infty]$, see figure \ref{fig:ppBA}. Moreover
\begin{equation}\label{eq:fflow}
f'(s) = \frac{ F((s,f(s)))_2 }{ F((s,f(s)))_1} = \frac{ f(s) - \varphi(s) } {f(s)}.
\end{equation}
The graph of the function $\varphi$ is the curve $L = \{ (s, \varphi (s))  : s \in \dR_+ \} $ and the set $$\Delta' = \{ (x_1, x_2) \in [0,\infty)^2 : x_2  < \varphi (x_1)\}$$ is the set of points below $L$. Assume that \eqref{eq:Xbox} does not hold. Then by \eqref{eq:Uyprime} and the intermediate value Theorem, the curves $L$ and $\Phi$ intersect. Then the exists $s_0 > 0$ such that
$$
f(s_0) = \varphi(s_0).
$$
From \eqref{eq:fflow}, $f'(s_0) = 0$ while $\varphi'(s_0) > 0$. It follows that $(s,f(s)) \in \Delta'$ for all $s \in (s_0, s_1)$ for some $s_1 > s_0$. Since $f'(s) < 0$ if $(s,f(s))\in \Delta'$  while $\varphi'(s) > 0$, the curves $L$ and $\Phi$ cannot intersect again. We get that all $s > s_0$, $(s,f (s)) \in \Delta'$.

On the other hand, since $f'(s) < 0$, for all $s > s_1$, $f(s) < f(s_1) < \varphi(s_1)$. If $x(t_1) = s_1$ and $\delta = \varphi(s_1) - f(s_1) >0$, we deduce that  for all $t > t_1$ , $x''(t) = x'(t) - \varphi(x(t)) < - \delta$. Integrating, this implies that $\lim_{t \to \infty} x'(t) = -\infty$ which contradicts \eqref{eq:Uyprime}.

We have proved so far that for all $t \geq 0$, $X(t) \in \Delta$. This implies that $x'(t)$ is increasing. In particular  $\lim_{t \to \infty} x(t) = \infty$ and $S = \infty$. Since $\lim_{s \to \infty} \varphi(s) = \lambda$, by \eqref{eq:xprime}, we readily deduce that $x'(t)$ converges to $\lambda$ as $t \to \infty$. \qed\end{proof}

\subsection{Comparison of second order differential equations}

It is possible to compare the trajectories of solutions of second order ODE by using the phase diagram. 
Let $\cD$ be the set of increasing  Lipschitz-continuous functions $\psi$ on $\dR_+$ such that  $\psi(0) = 0$. For two functions $\psi_1, \psi_2$ in $\cD$,  we write $\psi_1 \leq \psi_2$ if for all $t \geq 0$, $\psi_1 ( t) \leq \psi_2 (t)$.

\begin{lemma}\label{le:compode}
Let  $x \in \cH$ such that  $x = A( x ) $ and $x \not\equiv 0$. Let $\psi \in \cD$ and $y$ be a solution of  $y'' - y' + \psi(y) =0$  with  $y (0) = 0$, $y'(0) >0$. We define the exit times
$$
T = \inf\{ t \geq 0 : (y(t) , y'(t) ) \notin \dR^2_+ \},
$$
$$T_+ = \inf\{ t \geq 0 : y'(t) \geq \lambda \} \quad  \hbox{ and } \quad T_- = \inf\{ t \geq 0 : \varphi (y(t)) \leq y'(t)   \}.$$
\begin{enumerate}[(i)]
\item\label{odei2}
If $T_+ < T $, $T_+ < \infty$  and $\varphi \leq \psi$ then $y'(0)  \geq x'(0)$.

\item\label{odeii2}
If $T_- <  T$, $T_- < \infty$ and $\psi \geq \varphi$ then $x'(0)  \geq y'(0)$.
\end{enumerate}
\end{lemma}

\begin{proof}
Let us start with the hypothesis of \eqref{odei2}. The proof is by contradiction : we also assume that $y'(0)  <  x'(0)$. We set $Y(t) = (y(t), y'(t))$ and $G (y_1, y_2) = ( y_2 , y_2 - \psi (y_1))$. Define the trajectories $\Phi = \{ X(t) \in \dR_+^2 : t \geq 0 \}$, and for $\tau > 0$,
$\Psi(\tau) = \{ Y(t) \in \dR_+^2 : 0 \leq t \leq \tau \}$. By Lemma \ref{le:trajODE}, $\Phi$ is the graph of an increasing function $ f : \dR_+ \to \dR_+$ with $f(0) = x'(0) > 0$ and
$$
\Phi = \{ (s,f(s)) : s \in \dR_+ \}.
$$
Similarly, if $ t \in [0,T)$, $y'(t) > 0$. Thus, there exists a differentiable function $g :   [0 , y(T)] \to \dR_+$ such that
$$
\Psi (T) = \{ (s,g(s)) : s \in [0 , y(T)] \},
$$
with
$$
g'(s)  =  1 - \frac{ \psi(s) } {g(s)} .
$$
Now, the assumption $0 < y'(0) < x'(0)$ reads $0 < g(0) < f(0)$. Since $T_+ < T$, for $s \in [0,T_+]$, $g(s) > 0$ and there is a time $s_0 > 0$ such that $g(s_0)  \geq  \lambda$. In particular, by \eqref{eq:Uyprime}, $ f(s_0) < g(s_0)$. Hence, by the intermediate value Theorem, there exists a first time $0 < s_1 < s_0$ such that the curves intersect: $g(s_1)  = f(s_1)$ and $g(s)  <  f(s)$ on $[0, s_1)$. However, it follows from \eqref{eq:fflow} and  $\varphi \leq \psi$ that for $ s \in [0, s_1)$,
$$
g'(s) = 1 -  \frac{ \psi(s) } {g(s)}  \leq  1 -  \frac{ \varphi(s) } {g(s)} <  1 - \frac{  \varphi(s) } {f(s)} = f'(s).
$$
Hence, integrating over $[0,s_1]$ the above inequality gives 
$$g(s_1) - g(0) = \int_0^{s_1} g'(s) ds <  \int_0^{s_1} f'(s) ds  = f(s_1) - f(0).$$ 
However, by construction, $f(s_1) = g(s_1)$.  Thus, the above inequality contradicts $g(0) < f(0)$ and we have proved \eqref{odei2}. The proof of \eqref{odeii2} is identical and is omitted.
\qed\end{proof}

\subsection{Proof of Theorem~\ref{th:mainex}}
\label{subsec:prfmainex}
We first linearize \eqref{eq:odex} with $\varphi (s) = \lambda ( 1 - e^{-s})$
in the neighborhood of $\lambda= 1/4$.

\paragraph{Step one : Linearization from above.}
We have $\varphi'(0) = \lambda$, and from the concavity of $\varphi$,
\begin{equation}\label{eq:phiLB}
\varphi(s) \leq \lambda s .
\end{equation}
We take $\lambda > 1/4$ and consider the linearized differential equation
\begin{equation}\label{eq:odey}
y'' - y ' + \lambda y= 0.
\end{equation}
The solutions of this differential equation are
$$
y ( t ) = a \sin ( \omega t ) e^{\frac t 2} + b \cos ( \omega t ) e^{\frac t 2},
$$
with
$$
\omega = \sqrt { \lambda  - \frac 1 4}.
$$
We use this ODE to upper bound $x' (0)$ if $x = A(x)$. Recall that $A$ depends implicitly on $\lambda$.  
\begin{lemma}\label{le:UBx}
For any $\lambda > 1/4$, let  $x \in \cH$ such that  $x = A( x  ) $ and $x  \not\equiv 0$. We have
$$
x' (0) \leq  \frac{e^{2}}{4} e^{ - \frac{ \pi }{2 \omega} }  (  1+ O (\omega^2)).
$$
\end{lemma}

\begin{proof}
Let $a > 0$ and consider the function 
\begin{equation}\label{eq:deffyy}
y (t) =  a \sin ( \omega t ) e^{\frac t 2}.\end{equation} We have 
$y(0) = 0$, $y'(0) = a \omega$,
\begin{equation}\label{eq:patchy'}
y'(t) =  a e^{\frac t 2} (  \omega \cos (\omega t) + \frac 1 2 \sin ( \omega t ) ),
\end {equation} and
$$
y''(t) = a e^{\frac t 2} \PAR{ \omega \cos (\omega t ) + \PAR{ \frac 1 4  - \omega^2 } \sin (\omega t) }.
$$
Define
\begin{equation}\label{eq:ggtau}
\tau= \frac {\pi}{\omega} -  \frac 1 {\omega}
\arctan \PAR{ \frac {\omega}{\frac 1 4 - \omega^2}} = \frac {\pi}{\omega} - 4 + O (\omega^2).
\end{equation}
On the interval $[0, \tau]$, $y''(t) \geq 0$ and $y''(\tau) = 0$. Thus the function $y'(t)$ is increasing on $[0,\tau]$. Moreover, since $\cos(\omega \tau) =  - 1 + O ( \omega^2)$ and $\sin(\omega \tau) =  4 \omega + O ( \omega^3)$, we get from \eqref{eq:patchy'} that
$$
y'(\tau) = e^{-2} a e^{\frac{ \pi }{2 \omega} } (  \omega + O (\omega^3)).
$$
Using \eqref{eq:ggtau}, we have $\exp ( \tau /2) =\exp( \pi / (2\omega) - 2) ( 1  + O  ( \omega^2))$. Hence, we may choose $a$ in \eqref{eq:deffyy} such that $y'(\tau) = \lambda = \frac 1  4 + \omega^2$ with
$$
a = \frac{e^{2}}{4}  \frac{ e^{ - \frac{ \pi }{2 \omega} } }{ \omega} (  1+ O (\omega^2)).
$$
From what precedes, on the interval $[0,\tau]$, $y(t) > 0$ and $y'(t) >0$. From \eqref{eq:phiLB}, we may thus use Lemma \ref{le:compode}$(i)$ with $T_+ = \tau$ and $\psi(s) = \lambda s$. We get $x'(0) \leq y'(0) = a \omega$. 
\qed\end{proof}

\paragraph{Step two : linearization from below.}

For $0 < \kappa< 1$, we define
$$
\ell=  \frac 1 4 + \kappa^2 \omega^2 < \lambda,
$$
and the function in $\cD$
$$
\psi ( s) = \min \left( \ell s  , \varphi(s) \right).
$$
In particular
\begin{equation}\label{eq:phiUB}
\varphi  \geq  \psi.
\end{equation}
We shall now consider the linear differential equation
\begin{equation}\label{eq:odey2}
y'' - y ' + \ell y= 0,
\end{equation}
The solutions of \eqref{eq:odey2} are
$$
y ( t ) = a \sin ( \omega \kappa t ) e^{\frac t 2}  + b \cos ( \omega \kappa t ) e^{\frac t 2}.
$$
A careful choice of $a,\kappa$ will lead to the following lower bound.
\begin{lemma}\label{le:LBx}
For any $\lambda > 1/4$, let  $x \in \cH$ such that  $x = A( x  ) $ and $x  \not\equiv 0$. We have
$$
x' (0) \geq  \frac{ 8 e}{\pi} \omega^3  e^{ - \frac{ \pi }{2 \omega  } }  (  1+ O (\omega^2)).
$$
\end{lemma}

\begin{proof}
For $a> 0$, we look at the solution 
\begin{equation}\label{eq:defgfuegy}
y (t) =  a \sin ( \omega \kappa t ) e^{\frac t 2}.
\end{equation}
We have $y(0) = 0$, $y'(0) = a  \kappa \omega$.
$$y'(t) =  a e^{\frac t 2} (  \omega \kappa \cos (\omega \kappa t) + \frac 1 2 \sin ( \omega\kappa  t ) ).$$
We repeat the argument of Lemma \ref{le:UBx}. On the interval $[0, \tau]$, $y''(t) \geq 0$ and $y''(\tau) = 0$, where
\begin{equation}\label{eq:deffetau}
\tau= \frac {\pi}{\omega \kappa} -  \frac 1 {\omega\kappa }
\arctan \PAR{ \frac {\omega\kappa }{\frac 1 4 - \omega^2 \kappa^2}} = \frac {\pi}{\omega \kappa } - 4 + O (\omega^2),
\end{equation}
and the $O(\cdot)$ is uniform over all $\kappa > 1/2$. The function $y'(t)$ is increasing on $[0,\tau]$ and
$$
y'(\tau) =  a e^{-2}  e^{\frac{ \pi }{2 \omega} }  \omega \kappa ( 1 +  O (\omega^2)).
$$
Now, we have $\ell s   \leq  \varphi(s) $ for all $ s \in [0, \si]$ with
$$
\ell \si = \lambda ( 1 - e^{-\si}).
$$
It gives
$$
\si = 2  \PAR{ 1 - \frac{ \ell } {\lambda}  } + O \PAR{ 1 - \frac{ \ell } {\lambda}  }^2 = 8 ( 1 - \kappa^2)
\omega^2 + O ( ( 1 - \kappa^2) \omega^4).
$$ However from \eqref{eq:odey2}, for $t = \tau$, since $y''(\tau) = 0$, we have
$$
\frac{ y'(\tau) } { y (\tau) } =  \ell = \frac 1 4 + \kappa^2 \omega^2.
$$
From \eqref{eq:deffetau}, we have $\sin (\omega \kappa \tau) = 4 \omega \kappa + O (\omega^3)$ and $\exp ( \tau /2) = \exp ( \pi / (2\omega \kappa) - 2) ( 1 + O ( \omega^2))$. In \eqref{eq:defgfuegy}, we may thus choose $a$ such that $y(\tau) = \si$ by setting
$$
a = \si \frac{e^{2}}{4}  \frac{ e^{ - \frac{ \pi }{2 \omega \kappa} } }{ \omega \kappa} (  1+ O ( \omega^2)) = 2 e^{2}  e^{ - \frac{ \pi }{2 \omega \kappa} }  \frac{( 1 - \kappa^2) \omega }{ \kappa} (  1+ O ( \omega^2)) .
$$
Now, in the domain $0 \leq y \leq \si$, $\psi(y) = \ell \si$ and the non-linear differential equation $y'' - y ' + \psi(y)$ obviously
coincides with \eqref{eq:odey2}.  Thus, using  \eqref{eq:phiUB} and Lemma \ref{le:compode}$(ii)$ with $T_- = \tau$, it leads to
$$
x'(0)  \geq  y'(0) = a \kappa \omega = 2 e^{2}  e^{ - \frac{ \pi }{2 \omega \kappa} } ( 1 - \kappa^2) \omega^2   (  1+ O ( \omega^2)).
$$
Taking finally $\kappa = 1 - 2 \omega / \pi$ gives the statement.
\qed\end{proof}

\paragraph{Step three : End of proof.}

We now complete the proof of Theorem \ref{th:mainex}. We should consider the function $x(t)$ defined by \eqref{eq:defx}. We have seen that $x = A(x)$ and $x \not\equiv 0$ if $\lambda >1/4$. We start with the left hand side
inequality. From \eqref{eq:Xbox} in Lemma \ref{le:trajODE}, $x'(t)$ is increasing and we have
$$
x(t ) \geq x '(0) t.
$$
It follows from \eqref{eq:qQ} that
$$
q(\lambda ) = \int_0 ^\infty e ^ { - x (t)  } e^{ -t } dt \leq  \int_0 ^\infty e ^ { - x'(0) t  } e^{ -t } dt = \frac{ 1 } { 1 + x '(0)}.
$$
It remains to use Lemma \ref{le:LBx} and we obtain the left hand side of Theorem \ref{th:mainex}.

We now turn the right hand side inequality. For $X = (x_1, x_2) \in \dR^2$, define $G(X) = ( x_2, x_2 )$.  From
the definition of $F$ in \eqref{eq:defFX}, we have,  component-wise, for any $X \in \dR^2$,
$$
F(X) \leq G(X).
$$
Note also that $G$ is monotone : if component-wise $X \leq Y$ then $G(X) \leq G(Y)$. A vector-valued extension of Gronwall's inequality implies that if $X(0 ) = Y(0)$, $X' = F(X)$ and $Y' = G (Y)$ then, component-wise,
$$
X(t) \leq Y(t),
$$
(see e.g.  \cite[Exercise 4.6]{MR0171038}). Looking at the solution of $y'' - y' = 0$ such that $y(0) = 0$ and $y'(0 ) = x'(0)$, we get that
$$
x(t) \leq x'(0) (e^{t} - 1).
$$
We deduce from \eqref{eq:qQ} that, for any $T > 0 $,
\begin{align*}
q(\lambda)   = \int_0 ^\infty e ^ { - x (t)  } e^{ -t } dt 
& \geq \int_0 ^T e ^ { - x'(0) (e^{t} - 1) } e^{ -t } dt \\
& \geq  \int_0 ^T  ( 1 - x'(0) (e^{t} - 1) )  e^{ -t } dt \\
& \geq 1 - e^{-T}  - x'(0) T.
\end{align*}
Now, we notice that in order to prove  Theorem \ref{th:mainex},  by Lemma \ref{le:UBx}, we
may choose $\lambda$ close enough to $1/4$ so that $x'(0) < 1$. We finally take $T = -  \ln ( x'(0)  )$ and
apply Lemma \ref{le:UBx}. This concludes the proof of Theorem \ref{th:mainex}. \qed

\section{Proof of Theorem~\ref{th:RSstab}}

\label{pf1}

We define  the set recovered and infected vertices as $R(t) = \{ v \in V : X_v (t) = R \}$ and $I (t) =  \{ v \in V : X_v (t) = I \}$. The set $R(t)$ being non-decreasing, we may define $R(\infty) = \cup_{ t > 0} R(t)$ and $Z = | R(\infty) |  \in \dN \cup \{ \infty \}$. Note that also a.s. $R(\infty) = \{ v \in V : \exists t > 0 , X_v(t) = I \}$, in words,  $R(\infty)$ is the set of vertices which have been infected at some time.

Throughout this section, the chase-escape process is constructed thanks to i.i.d. $\EXP(\lambda)$ variables $(\xi_v)_{v \in V}$ and independent  i.i.d. $\EXP(1)$ variables $(D_v)_{v \in V}$. The variable $\xi_v$ (resp. $D_v$) is the time by which $v \in V$ will be infected (resp. recovered) once its ancestor is infected  (resp. recovered). 

\subsection{Subcritical regime}

We fix $0 < \lambda < \lambda_1$. In this paragraph we prove that $q_T (\lambda) =1$ if $T$ has upper growth rate at most $d$.  It is sufficient to prove that 
$\dE_\lambda Z $. To this end, we will upper bound the probability that $v \in R(\infty)$ for any $v \in V$. Let $V_k$ be the set of vertices of $V$ which are at distance $k$ from the root $\o$.  Let $v \in V_n$ and $v_0, \cdots, v_n$ be the ancestor line of $v$: $v_0 = \o$ and $v_{n} = v$. The vertex $v$ will have been infected if and only if for all $1 \leq m \leq n$, $v_{m-1}$ has infected $v_{m}$ before being recovered. We thus find
$$
\dP_\lambda ( v \in R(\infty) ) = \dP_\lambda \PAR{ \forall 1 \leq m \leq n, \sum_{i = 1} ^ m \xi_{v_i} < \sum_{i=1} ^m D_{v_{i-1}} } \leq \dP_\lambda \PAR{ \sum_{i = 1} ^ n \xi_{v_i} < \sum_{i=1} ^n  D_{v_{i -1}} }.
$$
The Chernov bound gives for any $0 < \theta < 1$,
\begin{eqnarray*}
 \dP_\lambda \PAR{ \sum_{i = 1} ^ n \xi_{v_i}< \sum_{i=1} ^n D_{v_{i-1}} } & \leq &  \dE_\lambda \exp \BRA{ \theta  \PAR{ \sum_{i=1} ^n D_{v_{i-1}} - \sum_{i = 1} ^ n \xi_{v_i}  } } \\
& = &  \PAR{\frac{ 1 } { 1 - \theta } }^n \PAR { \frac { \lambda } { \lambda + \theta } }^n ,
\end{eqnarray*}
where we have used the independence of all variables at the last line. Now, the above expression is minimized for $\theta = (1 - \lambda ) /2 > 0$ (since $\lambda < \lambda_1 < 1$). We find
$$
\dP_\lambda ( v \in R(\infty) )  \leq \PAR{\frac{  4 \lambda } {(\lambda +1)^2  } }^n.
$$
Also, from the growth-rate assumption, there exists a sequence $\veps_n \to 0$, 
\begin{equation*}\label{eq:cardVn}
| V_n | \leq   ( d + \veps_n )^n.
\end{equation*}
It follows that
$$
\dE_\lambda Z  = \sum_{v \in V} \dP_\lambda ( v \in R(\infty) )  \leq \sum_{n = 0} ^\infty \PAR{\frac{  4 (d+\veps_n) \lambda } {(\lambda +1)^2 } }^n.
$$
It is now straightforward to check that
$$
\frac{  4 d \lambda } {(\lambda +1)^2 } < 1,
$$
if $\lambda < \lambda_1$. This concludes the first part of the proof.

\subsection{Supercritical regime}

We now fix $\lambda > \lambda_1$. We should  prove that $q_T (\lambda)  < 1$ under the assumption that $T$ is lower $d$-ary. We are going to construct a random subtree of $T$ whose vertices are elements of $R(\infty)$ and which is a supercritical Galton-Watson tree.

First observe that we can couple two chase-escape processes with intensities $\lambda > \lambda'$ on  the same probability space in such a way that they share the same variables $(D_v)_{v \in V}$ and for all $v \in V$, $\xi_v^{(\lambda)} \leq \xi_v^{(\lambda')}$  (for example, we take $\xi_v^{(\lambda')} = (\lambda/\lambda'  ) \xi_v^{(\lambda)}$). The event of non-extinction is easily seen to be non-increasing in the variables $(\xi_v)_{v \in V}$ for the partial order on $\dR_+ ^V$ of component-wise comparison. It follows that the function $\lambda \mapsto q_T(\lambda)$ is non-increasing. We may thus assume without generality that $\lambda_1 < \lambda < 1$. For \(\delta  >0,\) we define the function $g_{\delta  }$ by, for all $x > 0$,
\begin{eqnarray*}\label{g_ep}
g_{\delta }(x) & = &  \frac 1 x  -  \log\left(\frac{1}{x }\right) +
\frac{\lambda }{x} - \log\left(\frac{\lambda }{x}\right) -2  - \log (\delta) \\
& = & \frac{ 1 + \lambda } { x } + \log \PAR{ \frac{ x^2 } {\lambda \delta}  } -2.
\end{eqnarray*}
Taking derivative, the minimum of $g_\delta$ is reached at $c= (1 + \lambda) /2$. We deduce easily the following property of the function \(g_{d }\).
\begin{lemma} \label{g_lem} If $\lambda_1 < \lambda < 1$, \(\min_{x  > 0 } g_{d} (x)<0\).
\end{lemma}
By Lemma \ref{g_lem}, using continuity, we deduce that there exist $c > 0$ and $1 < \delta < d$ such that
$$
g_{\delta} (c) < 0.
$$
In the remainder of the proof, we fix such pair $(c, \delta)$.

\paragraph{Construction of a nested branching process.}

We fix an integer \(m \geq 1\) that we will be completely specified later on. We assume that $m$ is large enough such that $T^{*m}$ contains a $\lceil \delta^{m} \rceil$-ary subtree. We denote by $T'$ this subtree and by $\rho \in V$ its root. For integer $k \geq 0$, we define $V'_k$ as the set of vertices of generation $k$ in $T'$. Note that by assumption $$|V'_k | = \lceil \delta^{m} \rceil ^ k.$$ We may assume that the generation of $\rho$ in $T$ is larger than $m$. We denote $a (\rho) \in V$ the $m$-th ancestor of $\rho$ in $T$. For $z \in V'_k$ and $k \geq 1$, we denote by $a (z)  \in V'_{k-1}$ its ancestor in $T'$. For example, if $z \in V'_1$,  $a(z) = \rho$.

We now start a branching process as follows. We set \(\rho\) to
be the root of the process, $\cS_0 = \{ \rho \}$. For integer $k \geq 1$, we define recursively the offsprings of the $k$-th generation as
the set $\cS_k$ of vertices $z \in V'_k$ satisfying the following three conditions :
\begin{enumerate}
\item the vertex $a(z) \in V'_{k-1}$ belongs to $\cS_{k-1}$;
\item  $\sum_{ i =1} ^m  \xi_{v_i} \leq \frac{m}{c}$ where $(v_0, v_1,\cdots, v_m)$ is the set of the vertices on the path from $a(z)$ to $z$, $v_0 = a(z)$, $v_m = z$;
\item $\sum_{ i  = 1} ^{m} D_{v_{i-1}} \geq \frac{m}{c}$ with $(v_0, v_1,\cdots, v_m)$  as above.
\end{enumerate}

Thus for \(z \in V'_k,\)  such that its ancestor $a(z) \in \cS_{k-1}$, we have that
\begin{eqnarray}\label{eq:offspringnested}
\dP_\lambda  \PAR{ z \in {\cS}_k | \cS_{k-1} } =  \dP_\lambda \left(\sum_{i =1}^{m} \xi_{v_i} \leq \frac{m}{c}\right)\dP_\lambda \left(\sum_{i =1}^{m} D_{v_{i-1}} \geq \frac{m}{c}\right). 
\end{eqnarray}

Notice that by construction, the number of offsprings of $z \ne z'$ in $\cS_{k-1}$ are  identically distributed and  are independent. It follows that the process forms a Galton-Watson  branching process. In the next paragraph, we will check that this branching process is supercritical, i.e. we will prove that
\begin{equation}\label{eq:Msup}
M = \sum_{v \in V'_1} \dP_\lambda  \PAR{ z \in {\cS}_1  }  > 1.
\end{equation}
It implies that with positive probability, the branching process does not die out (see Athreya and Ney \cite[chapter 1]{MR2047480}).

Before proving \eqref{eq:Msup}, let us first check that it implies Theorem~\ref{th:RSstab}. Assume that at some time $t > 0$, the vertex $\rho$ becomes infected and that $a(\rho)$ is still infected. Assume further that $\sum_{i=1}^m D_{v_{i-1}} \geq m/c$, where $(v_0, v_1,\cdots, v_m)$ is the set of the vertices on the path from $a(\rho)$ to $\rho$. Note that the existence of such finite time $t > 0$ and such sequence $(D_{v_i})_{0 \leq i \leq m-1}$ has positive probability. Let us denote by $E$ such event. We set $t_0 = t$ and, for integer  $k \geq 1$, $$t_k = t_{k-1} +\frac{m}{c}.$$
By construction, if $E$ holds and $z \in \cS_k$ then, at time $t_k$, $z$ and $a(z)$ are both infected. Hence on the events of $E$ and of non-extinction of the nested branching process, the chase-escape process does not get extinct. It thus remains to prove that \eqref{eq:Msup} holds.

\paragraph{The nested branching process is supercritical.}

We need a standard large deviation estimate. We define
$$
J (x) = x - \log{x} - 1.
$$
The next lemma is an immediate consequence of Cramer's Theorem for exponential variables (see \cite[\S 2.2.1]{MR2571413}).
\begin{lemma}\label{lm1}
Let \((\zeta_i)_{ i \geq 1},\) be i.i.d. $\EXP(\lambda)$ variables. For any \(a > 1/\lambda,\)
we have that
$$\liminf_{m \to \infty} \frac{1}{m}\log\dP\left(\sum_{i = 1}^{m} \zeta_i \geq a m\right) \geq - J (\lambda a),$$
while, for any \(a < 1/\lambda,\)
\[\liminf_{m \to \infty} \frac{1}{m}\log\dP \left(\sum_{i = 1}^{m} \zeta_i \leq am\right) \geq - J(\lambda a).\] \end{lemma}

Note that the bounds of Lemma~\ref{lm1} hold for all $a > 0$ (even if they are sharp only for the above ranges).  We may now estimate the terms in \eqref{eq:offspringnested}. We have from Lemma~\ref{lm1} that
\begin{eqnarray*}
 \dP_\lambda \left(\sum_{i =1}^{m} \xi_{v_i} \leq \frac{m}{c}\right)& \geq & \exp\BRA{-m J \left(\frac{\lambda }{c}  \right) + o(m)} \nonumber
\end{eqnarray*}
and
\begin{eqnarray*}
\dP_\lambda \left(\sum_{i =1}^{m} D_{v_{i-1}} \geq \frac{m}{c}\right) & \geq &\exp\BRA{-m J\left(\frac{1}{c}  \right) + o(m)}.
\end{eqnarray*}
Thus we obtain a lower bound on the mean number of offspring in the first generation to be
\begin{eqnarray*}
M &=& \sum_{z \in V'_1 } \dP_\lambda  (z \in {\cS}_1 ) \nonumber\\
&\geq & \lceil \delta^{m} \rceil \exp\BRA{- m  \PAR{ J \left(\frac{1}{ c }\right)  + J\left(\frac{\lambda}{ c}\right)  + o(m) }}  \nonumber\\
&\geq & \exp\left(-m g_{\delta}(c)  + o(m)\right) \label{m_eq},
\end{eqnarray*}
where \(g_{\delta}(.)\) is as defined in (\ref{g_ep}). If $m$ was chosen large enough, we have
that \(M > 1\) and hence that the branching process is supercritical. Therefore with positive
probability, this branching process does not die out. This proves the theorem. \qed

\section{Proof of Theorem ~\ref{th:RSext}}

\label{sec:RSext}

The proof of Theorem  \ref{th:RSext} parallels the proof of Theorem  \ref{th:mainex}. Even if the strategy is the same, we will meet some extra difficulties in the study of the phase diagram (notably in the forthcoming Lemma \ref{le:compode00}).

\subsection{Differential equation for the survival probability}

We first determine a differential equation associated to the probability of extinction.  Under $\dP'_\lambda$, define $Q_\lambda (t)$ to be the extinction probability given that the root $\so$ is recovered at time $t  \geq 0$ so that
\begin{equation}\label{eq:qQ0}
q(\lambda) = \int_0 ^\infty Q_{\lambda} (t) e^{-t} dt
\end{equation}
and $Q_\lambda (0) =1$.

Now, in $T$, the offsprings of the root are $\{1,\cdots,N \}$, where $N$ has distribution $P$. The root infects each of its offspring after an independent exponential variable with intensity $\lambda$.  Let $\{\xi_i\}_{1 \leq i \leq N}$ be the infection
times. Note that in $T$, the subtrees generated by each of the offsprings of the root are iid copies of $T$.  Hence, if for integer $i$ with $1 \leq \xi_i \leq D_{\so}$, we define $X^i$ as the subprocess on vertices $(i \dN^f )\cap V$ with ancestors $i$. Conditioned on $D_{\so} = t $, on $N$ and  $(\xi_i)_{1 \leq i \leq N}$, the processes $(X^i)$ are independent chase-escape processes conditioned on the fact that root becomes at risk at time $t - \xi_i$ (where we say that a $I$-vertex is at risk if its genitor is in state $R$).

For the process $X$ to get extinct, all
the processes $X^i$ must get extinct.  So finally, we get
\begin{eqnarray*}
Q_\lambda ( t )  & =  & \dE'_\lambda \SBRA{ \prod_{1 \leq i \leq N } \PAR{ \IND ( \xi_ i > t )  +  \IND ( \xi_ i \leq t ) Q_\lambda ( t - \xi_i + D_i )} }
\end{eqnarray*}
where $(D_i), i \geq 1,$  are independent exponential variables with parameter $1$. Consider the generating function of $P$
$$
\psi ( x)  = \dE'_\lambda  \SBRA{ x ^ N }= \sum_{ k =0} ^ \infty x^k P (k).
$$
Recall that $\psi$ is strictly increasing and convex on $[0,1]$ and $\psi'(1) = \dE N = d$. We find,  for any $t \geq 0$,
\begin{eqnarray*}
Q_\lambda(t)  &=&
\psi  \PAR{ e^{-\lambda t } +  \lambda \int_0 ^t e^{-\lambda x } \int_0 ^ \infty Q_\lambda ( t - x + s) e^{-s} ds dx  }  \\
& = & \psi  \PAR{ e^{-\lambda t } +  \lambda e^{-\lambda t }\int_0 ^t e^{\lambda x } \int_0 ^ \infty Q_\lambda (  x + s) e^{-s} ds dx  } \\
& =  &  \psi  \PAR{ e^{-\lambda t } +  \lambda e^{-\lambda t }\int_0 ^t e^{(\lambda +1)x} \int_x ^ \infty Q_\lambda ( s) e^{-s} ds dx } .
\end{eqnarray*}
Performing the change of variable
\begin{equation}\label{eq:defx0}
x (t)  =   \psi^{-1}   ( Q_\lambda (t) ) \in [0,1],
\end{equation}
leads to
\begin{equation}
\label{eq:fpx0t}
  x(t) = e^{ - \lambda t } +  \lambda e^{- \lambda t }\int_0 ^t e^{(\lambda +1)x} \int_x ^ \infty \psi ( x( s)  ) e^{-s} ds dx.
\end{equation}
We multiply the above expression by $e^{ \lambda  t}$ and differentiate once, it gives
\begin{equation}\label{eq:xprime0}
e^{\lambda t } ( \lambda x (t) +  x' (t) )  =  \lambda  e^{(\lambda +1)t}  \int_t ^\infty \psi( x(s))  e^{-s} ds,
\end{equation}
Now, multiplying the above expression by $e^{- ( \lambda +1 ) t}$ and differentiating once again, we find that \(x(t)\)
satisfies the differential equation
\begin{equation}\label{eq:odex0}
x'' - ( 1 - \lambda) x ' +  \varphi(x) = 0
\end{equation}
with
$$
\varphi (x) =  \lambda  \psi(x) - \lambda x.
$$

\subsection{A fixed point equation}

We define $\rho \in [0,1)$ as the extinction probability in the Galton-Watson tree:
$$
\rho = \psi ( \rho).
$$
We note that $\varphi$ is convex, $\varphi ( 1)  = \varphi (\rho )= 0$, $\varphi$ is negative on $(\rho, 1)$ and it is increasing in a neighborhood of $1$,  $\varphi'(1) = \lambda ( d - 1) > 0$. The fact that $\varphi$ is not monotone is the main difference with the proof of Theorem \ref{th:mainex} .

Let $\cH$ be the set of non-increasing functions $f : \dR_+ \to \dR_+$ such that $f(0) = 1$, $\lim_{ t \to \infty } f(t)  =  \rho$. The next lemma is an easy consequence of the monotony of the process. 
\begin{lemma}
For any $\lambda > \lambda_1$, the function $x(\cdot)$ defined by \eqref{eq:defx0} is in $\cH$.
\end{lemma}
\begin{proof}
As in the previous section, we may construct the chase escape process conditioned on the root is recovered at time $t$ thanks to i.i.d. $\EXP(\lambda)$ variables $(\xi_v)_{v \in V}$ and independent  i.i.d. $\EXP(1)$ variables $(D_v)_{v \in V \ne \so}$ and set $D_\so = t$. The variable $\xi_v$ (resp. $D_v$) is the time by which $v \in V$ will be infected (resp. recovered) once its ancestor is infected  (resp. recovered). The event of extinction is then non-increasing in $t$. It follows that the map $t \mapsto Q_\lambda(t)$ is non-increasing. From \eqref{eq:defx0}, it follows that $x(t)$ is also non-increasing. We may thus define $a = \lim_{t \to \infty} x(t)$. 
Using the continuity of $\psi$ leads to
$$
 \lambda e^{- \lambda t }\int_0 ^t e^{(\lambda +1)x} \int_x ^ \infty \psi ( x( s)  ) e^{-s} ds dx =  \lambda e^{- \lambda t } \int_0 ^t e^{\lambda  x}  \psi ( a ) ( 1 + o(1) )  dx, 
$$
This last integral being divergent as $t \to \infty$, we deduce that 
$$
 \lambda e^{- \lambda t }\int_0 ^t e^{(\lambda +1)x} \int_x ^ \infty \psi ( x( s)  ) e^{-s} ds dx = \psi(a) + o(1). 
$$
From \eqref{eq:fpx0t}, we get that $a = \psi (a)$ which implies that $a \in \{ \rho, 1\}$.  Note however that Theorem \ref{th:RSstab} and Lemma \ref{le:daryGWT} imply that $q(\lambda) < 1$ for all $\lambda > \lambda_1$. Then \eqref{eq:qQ0} and the monotony of $t \mapsto Q_{\lambda}(t)$ give that for all $t \geq t_0$ large enough, $Q_{\lambda} (t) < 1$. From \eqref{eq:defx0} it implies in turn that for all $t \geq t_0$, $x(t) < 1$. So finally $a \leq x(t_0) < 1$ and $a = \rho$. \qed \end{proof}

From now on in this section, we fix a small $u >0$ and we assume that
\begin{equation}\label{eq:rangelambda}
\lambda_1 < \lambda < 1 - u.
\end{equation}
We define the map $A : \cH \to L^ \infty ( \dR_+ , \dR_+)$ defined by
\begin{equation}\label{eq:defT0}
A (y) (t) = e^{-\lambda t}  + \lambda e^{-\lambda t} \int_0 ^t e^{(\lambda +1)x} \int_x ^ \infty \psi ( y( s)  ) e^{-s} ds dx.
\end{equation}
Since
$ \max_{ x \in [0,1] } |\psi(x) | = 1$ it is indeed straightforward to check
that $A (y)$ is in $ L^ \infty ( \dR_+ , \dR_+)$: $A (y)(t)$ is
bounded by $1$. Note also that $y \equiv 1$ is a solution
of the fixed point equation
$$
y = A  (y).
$$
By \eqref{eq:fpx0t}, we find that the function $x$ defined by \eqref{eq:defx0} satisfies also the fixed point $x = A(x)$. In the
sequel, we are going to analyze the non trivial fixed points of $A$.

Let $x \in \cH$ such that $x = A (x)$. Then $x \not\equiv 1$. By induction, it
follows easily that $t \mapsto x(t)$ is twice differentiable. In particular, from the argument following \eqref{eq:fpx0t}, $x$ satisfies \eqref{eq:odex0} and we are looking for a specific non-negative solution of \eqref{eq:odex0} with $x(0) = 1$. To characterize completely this solution, it would be enough to compute $x'(0)$ (which is necessarily negative since $x(0) = 1$, $x'(0) = 0$ corresponds to the trivial solution $x \equiv 1$). We will perform this in the next subsection in an asymptotic regime. We start with some properties obtained  from the phase diagram of the ODE \eqref{eq:odex0}.

\begin{lemma}\label{le:Uyprime0}
Let $x \in \cH$ such that $x = A (x)$. Then,
\begin{enumerate}[(i)]
\item
 for all $t>0$, $\rho < x(t) < 1$; 
\item for all $t \geq 0$,  
$
- 1 <   x'(t) <  0.
$
\end{enumerate}
\end{lemma}
\begin{proof}
Let us prove $(i)$. We first observe that since $x(t)$ is non-increasing, $x(0) =1$ and $x'(0) < 0$, we have that for all $t > 0$, $x(t) < 1$. Also, if $x(t) = \rho$ for some $t >0$, then $x(s) = \rho$ for all $s \geq t$ (since $x$ is non-increasing and has limit $\rho$). However $y \equiv \rho$ being a distinct solution of \eqref{eq:odex0},  $x$ and $y$ cannot coincide on an interval. We thus have for all $t >0$,  $\rho < x(t) < 1$.

We now prove $(ii)$. Assume that there is a time $t > 0$ such that $x'(t) = 0$. Then, from  \eqref{eq:odex0} and $\rho < x(t) < 1$, we deduce that $x'' (t) > 0$. In particular, $x'(s) > 0$ for all $s \in (t , t + \delta)$ for some $\delta > 0$. This contradicts that $x(\cdot)$ is non-increasing. Also, from \eqref{eq:xprime0}, for any $t \geq 0$, $\lambda x(t) + x'(t) > 0$. Since $x(t) \leq 1$, we deduce that for all $t \geq 0$, $
- \lambda <   x'(t) <  0$. \qed
\end{proof}
We define $X(t) = (x (t) , x'(t))$  and $$F (x_1 , x_2)  = ( x_2 , ( 1 - \lambda ) x_2 - \varphi (x_1) ) $$ so that
\begin{equation}\label{eq:defFX0}
X ' = F (X).
\end{equation}
We define the trajectory $\Phi = \{ X(t) : t \geq 0 \}$. Recall that $\rho  = \lim_{ t \to \infty} x(t)$. Also, since for all $t \geq 0$, $X(t)'_1 = F(X(t))_1 < 0$, $\Phi$ is the graph of a differentiable function $ f : (\rho,1] \to (-1,0)$ with $f(1) =  x'(0) <  0$,
$$
\Phi = \{ (s,f(s)) : s \in (\rho,1] \}.
$$
Moreover
\begin{equation}\label{eq:fflow0}
f'(s) =   \frac{ F((s,f(s)))_2 }{ F((s,f(s)))_1} =   1 - \lambda - \frac{ \varphi(s) } {  f(s)  }.
\end{equation}

\begin{figure}
\centering \scalebox{0.6}{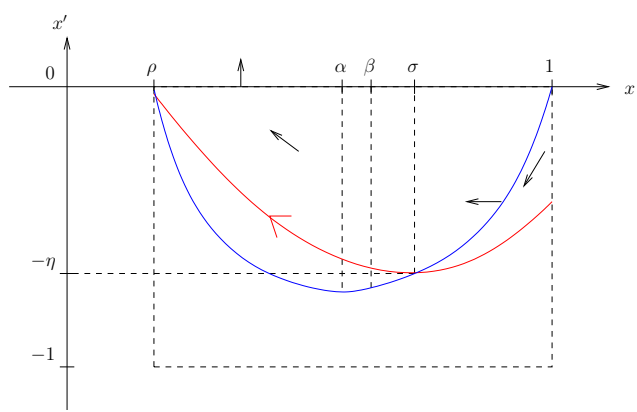}
\caption{Illustration of the phase portrait. In blue, the curve $\Gamma$, in red the curve $\Phi$.}\label{fig:pp}
\end{figure}

We notice that on the curve  $$\Gamma = \{ (x_1, x_2 ) \in [\rho, 1]\times [-1, 0] :  ( 1 - \lambda ) x_2 = \varphi (x_1) \}$$ the second coordinate of $F$ vanishes (see figure \ref{fig:pp}). The next lemma shows that our function $(x(t), x'(t))$ cannot cross $\Gamma$ near its origin $(x(0), x'(0))$.
\begin{lemma}\label{le:compode00}
There exists $\delta > 0$ depending only on $\psi$ and $u$ defined in \eqref{eq:rangelambda} such that  the following holds.  Let  $x \in \cH$ such that  $x = A( x)$. If $(x(t), x'(t) ) \in \Gamma$ for some $t > 0$, then $x'(t)   \leq  -  \delta$ and $x(t) \leq 1 - \delta$.
\end{lemma}

\begin{proof}
Define $\sigma$ as the largest $s$  such that $ (s , f(s) ) \in \Gamma$ (see figure \ref{fig:pp}). We have $\sigma < 1$. We should prove that $\sigma \leq 1 - \delta$ and $f(\sigma) \leq - \delta$. Recall that $f(1) = x'(0) < 0$. Thus, on $(\sigma,1]$, $(s,f(s))$ is below $\Gamma$ and it follows that $f$ is increasing. By construction, $f(\sigma) = \varphi(\sigma)/(1- \lambda) = \lambda ( \psi(\sigma) - \sigma)/(1- \lambda) $ and $f'(\sigma) = 0$. We deduce that 
\begin{equation}\label{eq:fsecond}
f''(\sigma) = -\frac{ \varphi'(\sigma)}{  f(\sigma) }+ \frac{f'(\sigma) \varphi(\sigma) } {f^2 (\sigma)} =  - \frac{\varphi'(\sigma) }{ f(\sigma)} \geq 0,
\end{equation}
where the last inequality comes from $f$ is increasing on  $[\sigma,1]$ and $f'(\sigma) = 0$.

We define $\alpha \in (\rho, 1)$ as the point where the function $\kappa(x) = x - \psi(x)$ reaches its maximum. Since $\varphi'(s) < 0$ on $[\rho, \alpha)$,  from \eqref{eq:fsecond} we find that $\sigma \in [\alpha, 1)$. We set $f(\sigma) = - \eta$. We will prove  that $\eta \geq \lambda \delta_0 / (1 - \lambda)$ for some $\delta_0 >0$ depending only on $u$ and $\psi$. This will conclude the proof of the lemma. Indeed, by construction $ ( 1 - \lambda) \eta = - \varphi( \sigma) = \lambda( \sigma - \psi(\sigma)) $. The function $\kappa(x) = x - \psi(x)$ has a continuous decreasing inverse in $[\alpha,1]$ with $\kappa^{-1} (0) = 1$.  Hence, $\sigma = \kappa^{-1} (( 1 - \lambda) \eta / \lambda)$ and, if $\eta \geq \lambda \delta_0 / ( 1- \lambda)$, we deduce that the statement of the lemma holds with $\delta = \min(\lambda_1 \delta_0 / ( 1 - \lambda_1), 1 - \kappa^{-1} (  \delta_0 ))$.

To this end, we fix any $\beta \in ( \alpha , 1)$, we set $ b =  \kappa(\beta) > 0$, and
\begin{equation}\label{eq:choicedelta0}
\delta_0 = \min \left( \frac b 2 ,   (1 - u) \sqrt{ \frac{ b ( \beta - \alpha)  }{ u }} \right). 
\end{equation}
We assume that $\eta < \lambda \delta_0 / ( 1- \lambda)$ and look for a contradiction. We first notice that $\delta_0 < b$ implies that $\sigma = \kappa^{-1} ( ( 1 -  \lambda )  \eta /  \lambda ) >  \kappa^{-1} ( b  ) = \beta$.

Consider the solution $Y(t) = ( y (t) , y'(t) )$ of the ODE \eqref{eq:defFX0} with initial condition $Y(0) = ( \beta , - \eta)$. The trajectory of $Y(t)$ is denoted by $\tilde \Phi = \{ Y(t ) : t \geq 0\}$. We define the set $\Gamma_+ = \{( x_1, x_2 ) \in [\alpha, 1)\times [ -\eta , 0) :( 1 - \lambda ) x_2 \geq \varphi (x_1) \}$. On $\Gamma_+$, $F(x) _1 < 0 $ and  $F(x) _2 \geq 0$. It follows that the trajectories $\Phi$ and $\tilde \Phi$ exit $\Gamma_+$ either on its left side  $\{(\alpha, x_2) , x_2 \in [-\eta , 0] \}$ or its upper side $\{(x_1, 0) , x_1 \in [\alpha , 1) \}$. However, Lemma \ref{le:Uyprime0}$(ii)$ implies that $\Phi$ exits $\Gamma_+$ on the left side. Since $\Phi$ and $\tilde \Phi$ cannot intersect and $\tilde \Phi$ is on the left side of $\Phi$  in $\Gamma_+$ (since $\beta < \sigma$), we deduce that necessarily, $\tilde \Phi$ also exits $\Gamma_+$ on the left side. We now check that, with our choice of $\delta_0$ in \eqref{eq:choicedelta0}, it contradicts $\eta < \lambda \delta_0 / ( 1- \lambda)$.

Define $\tau >0$ as the exit time of $Y(t)$ from $\Gamma_+$. If $0 \leq t \leq \tau$, using that $\varphi$ is increasing on $[\alpha, \beta]$, we find 
$$y'' (t)  \geq - ( 1 - \lambda) \eta -  \varphi( \beta )  \geq  - \lambda \delta_0 + \lambda b \geq \lambda b /2, 
$$ since $\delta_0 \leq b /2$. We deduce that for all $t \in [0, \tau]$, $y' (t) \geq z'(t) $ and $y(t) \geq z(t) $ with $z(t) =  (\lambda b/4) t^2 - \eta t + \beta$. We set $t_e = 2 \eta / (\lambda b)$. Since $z'(t_e) = 0$, we have $\tau \leq t_e$. Also $z$ being decreasing on $[0,t_e]$, we have $y(\tau) \geq z(t_e) = - \eta^2 / ( \lambda b ) + \beta$. Thanks to \eqref{eq:choicedelta0},  $\eta^2 < \lambda^2 \delta_0^2 / (1-\lambda)^2 < \lambda b ( \beta - \alpha)$ and we deduce that $ y(\tau) \geq - \eta^2 / ( \lambda b ) + \beta > \alpha$. In particular, $\tilde \Phi$ exits $\Gamma_+$ on the upper side. It leads to a contradiction. We have thus proved that $\eta \geq \lambda \delta_0 / (1 - \lambda)$.  \qed\end{proof}

\subsection{Comparison of second order differential equations}

For two  functions $\varphi_1, \varphi_2$ on $[0,1]$, we write $\varphi_1 \leq \varphi_2$ if for all $ t \in [0,1]$, $\varphi_1 ( t) \leq \varphi_2 (t)$. The next lemma is proved as Lemma \ref{le:compode}, we omit its proof.
\begin{lemma}\label{le:compode0}
Let $\delta > 0$ be as in Lemma \ref{le:compode00}. Let  $x \in \cH$ such that  $x = A x$. Let $\tilde \varphi$ be a Lipshitz-continuous function and $y$ be solution of $y'' - ( 1 - \lambda ) y' + \tilde \varphi(y) =0$ with $y (0) = 1$, $y'(0) < 0$. We define the exit times
$$
T = \inf\{ t \geq 0 : (y(t) , y'(t) ) \notin [0,1] \times [-1, 0]  \},
$$
$$T_-= \inf\{ t \geq 0 : y'(t) \leq  -1  \} \quad  \hbox{ and } \quad T_+ = \inf\{ t \geq 0 : ( 1 - \lambda) y'(t) =  \varphi ( y(t) )  , �y(t)  \geq 1 - \delta \}.$$
\begin{enumerate}[(i)]
\item\label{odei}
If $T_+  <  T  < \infty$  and $\tilde \varphi \geq  \varphi $ then $y'(0)   \geq  x'(0)$.

\item\label{odeii}
If $T_-  =   T  < \infty$ and $\tilde \varphi  \leq \varphi$ then $x'(0)  \geq y'(0)$.
\end{enumerate}
\end{lemma}

\subsection{Proof of Theorem~\ref{th:RSext}}

The proof of Theorem~\ref{th:RSext} follows now closely the proof of Theorem  \ref{th:mainex}. We first linearize \eqref{eq:odex0}  in the neighborhood of $\lambda_1$.

\paragraph{Step one : linearization from below.}

We have $\varphi ( 1) = 0$,  $\varphi'(1) = \lambda (d -1 ) > 0$, and from the convexity of $\varphi$,
\begin{equation}\label{eq:phiLB0}
\varphi(s) \geq \lambda (d -1 )  ( s -1 )  .
\end{equation}
We take $\lambda_1 < \lambda < 1 $ and consider the linearized differential equation
\begin{equation}\label{eq:odey0}
y'' -  (1 - \lambda ) y ' + \lambda (d -1 ) (y -1 ) = 0.
\end{equation}
The solutions of this differential equation are
$$
y ( t ) =  1 + a \sin ( \omega t ) e^{ \frac{(1 -\lambda ) t}{2}} + b \cos ( \omega t ) e^{\frac{(1 -\lambda ) t}{2}},
$$
where
$$
\omega =\frac 1 2 \sqrt {  - \lambda^2 + 2  (2 d -1 ) \lambda  - 1 } =c (\lambda) \sqrt { \lambda  -\lambda_1 },
$$
and
$$
c (\lambda ) =\frac 1 2  \sqrt { ( 2d -1 ) + 2 \sqrt {d (d -1)  }  - \lambda } =   \left( d (d -1)  \right)^{1/4} + O ( | \lambda - \lambda_1 | ) .
$$
We use this ODE to bound from below $x' (0)$ if $A(x) = x$.
\begin{lemma}\label{le:UBx0}
The exists a constant $c_0 > 0$ such that for any $\lambda_1 < \lambda < 1 $, if $x \in \cH$ satisfies $A(x) = x$ then 
$$
x'(0) \geq  -  c_0  e^{ - \frac{ \pi ( 1 - \lambda) }{2 \omega} }    (  1+ O (\omega^2)).
$$
\end{lemma}

\begin{proof}
We can assume without loss of generality that $\lambda$ satisfies \eqref{eq:rangelambda}. Let $a  < 0$, $b =  (1 -\lambda) /2$, and consider the function $$y (t) =  1 +  a \sin ( \omega t ) e^{ bt }.$$ We have $y(0) = 1$, $y'(0) = a \omega$,
$$y'(t) =   a e^{ b t} (  \omega \cos (\omega t) + b \sin ( \omega t ) ), $$ and
$$
y''(t) = a e^{b t} \PAR{  2 b  \omega \cos (\omega t ) + \PAR{ b ^2  - \omega^2 } \sin (\omega t) }.
$$
Define
$$
\tau = \frac {\pi}{\omega} -  \frac 1 {\omega}
\arctan \PAR{ \frac { 2 b \omega}{ b ^2 - \omega^2}} = \frac {\pi}{\omega} - \frac{2}{b}  + O (\omega^2).
$$
On the interval $(0, \tau)$, $y''(t) < 0$ and $y''(\tau) = 0$. Thus the function $y'(t)$ is decreasing on $[0,\tau]$ and
$$
y'(\tau) =  e^{ - \frac{ 2}{ b }}  a e^{\frac{ \pi b  }{ \omega} } (  \omega + O (\omega^3)).
$$
Hence, we may choose $a$ such that $y'(\tau) =  -1 $ with
$$
a =  -  \omega^{-1} e^{ \frac{ 2}{ b }}    e^{ - \frac{ \pi  b }{ \omega} }  (  1+ O (\omega^2)).
$$
It remains to use \eqref{eq:phiLB0} with Lemma \ref{le:compode0}$(ii)$ and $\tau = T_-$.
\qed\end{proof}

\paragraph{Step two : linearization from above.}

For $ 0 < \eta < \min ( 1 , c (\lambda) / (d-1) ) $, we define
$$
\ell =  ( 1 - \eta ) \lambda + \eta \lambda_1  < \lambda,
$$
and the Lipschitz-continuous function
$$
\tilde \varphi ( s) = \max \left(  \varphi  (s) ,  \ell ( d - 1) ( s - 1) \right).
$$
In particular
\begin{equation}\label{eq:phiUB0}
\varphi  \leq  \tilde \varphi.
\end{equation}
We define the linear differential equation
\begin{equation}\label{eq:odey20}
y'' - ( 1 - \lambda ) y ' + \ell (d - 1) (y  - 1) = 0.
\end{equation}
The solutions of \eqref{eq:odey20} are
$$
y ( t ) =  1 + a \sin ( \omega' t ) e^{ \frac{(1 -\lambda ) t}{2}} + b \cos ( \omega' t ) e^{\frac{(1 -\lambda ) t}{2}},
$$
with
$$
\omega'  =\frac 1 2 \sqrt {  - \lambda^2 + 2  (2 d -1 ) \lambda  - 1 - 4 \eta ( d- 1) ( \lambda - \lambda_1)  } = \omega \sqrt { 1 - \frac{ \eta  ( d-1) }{ c( \lambda) }}.
$$
In the sequel, $o(1)$ denotes a function which goes to $0$ as $\omega$ goes to $0$.
\begin{lemma}\label{le:LBx0}
If $P$ has finite second moment, then there exists a constant $c_1 > 0$ such that for all $\lambda_1 < \lambda < 1$,  if $x \in \cH$ satisfies $A(x) = x$ then 

$$
x' (0) \leq  - c_1  \omega^3  e^{ - \frac{ \pi  ( 1 - \lambda) }{ 2 \omega  } }  (  1+ o (1)).
$$
\end{lemma}

\begin{proof}
We can assume without loss of generality that $\lambda$ satisfies \eqref{eq:rangelambda}. We set $$ b = \frac{ 1- \lambda}{ 2}  \quad  \hbox{  and } \quad \kappa = \sqrt { 1 - \frac{ \eta  ( d-1) }{ c( \lambda) }}.$$ We parametrize in terms of $\kappa$, so that
\begin{equation}\label{eq:elllambda0}
\ell = \lambda - ( 1 - \kappa^2) \omega^2  \quad \hbox{ and } \quad \omega' = \kappa \omega.
\end{equation}
For $a<  0$, we look at the solution $$y (t) = 1 +  a \sin ( \omega \kappa t ) e^{ b t }.$$
We have $y(0) = 1$, $y'(0) = a  \kappa \omega$.
$$y'(t) =  a e^{ b t} (  \omega \kappa \cos (\omega \kappa t) +  b  \sin ( \omega\kappa  t ) ).$$
We repeat the argument of Lemma \ref{le:UBx0}. On the interval $[0, \tau]$, $y''(t) \geq 0$ and $y''(\tau) = 0$, where
$$
\tau = \frac {\pi}{\omega \kappa} -  \frac 1 {\omega\kappa }
\arctan \PAR{ \frac { 2 b \omega\kappa }{b ^2 - \omega^2 \kappa^2}} = \frac {\pi}{\omega \kappa } - \frac{2}{b}  + O (\omega^2),
$$
and the $O(\cdot)$ is uniform over all $\kappa > 1/2$. The function $y'(t)$ is increasing on $[0,\tau]$ and
$$
y'(\tau) = e^{-2} a e^{\frac{ \pi  b }{ \omega \kappa } } (  \omega  \kappa + O (\omega^2)).
$$
Now, we have $\varphi  (s) <  \ell ( d - 1) ( s - 1) $ for all $ s \in [1 - \si, 1]$ with
$$
 - \ell (d-1) \si = \varphi ( 1 - \si) = \lambda ( \psi ( 1 - \si ) - 1 + \si ).
$$
If $P$ has finite second moment then, from Abel's Theorem, $\psi''$ is continuous on $[0,1]$. Also from Jensen's inequality, $\psi''(1) \geq d(d - 1) > 0$. We expand $\psi$ in a neighborhood of $1$, as $\omega \to 0$,  it leads to,
\begin{eqnarray*}
\si  & = & \frac{ 2 (d -1)}{ \psi'' (1) \lambda }  \PAR{ \lambda -  \ell   } ( 1+ o (1) ) \\
&  = & \frac{ 2 (d -1)}{ \psi'' (1) \lambda }  ( 1 - \kappa^2)
\omega^2 ( 1 + o (1) ),
\end{eqnarray*}
where $o(1)$ is uniform over all $0 < \kappa < 1$. In particular, for all $\omega$ small enough, $\si < \delta$ with $\delta$ as in Lemma \ref{le:compode00}. Also, from \eqref{eq:odey20}, for $t = \tau$, since $y''(\tau) = 0$, we have
$$
  \frac{  y (\tau) - 1 } { y'(\tau) }  =   \frac{  1- \lambda }{ \ell (d-1) } =     \frac{  2 b  }{ \ell (d-1) }.
$$
We may choose $a$ such that $y(\tau) =  1 - \si$ by setting
$$
a =  - \si  e^{2} \frac{ \ell ( d- 1) }{ 2 b   }  \frac{ e^{ - \frac{  \pi b }{ \omega \kappa} } }{ \omega \kappa}  (  1+ O ( \omega^2)) =  - e^{2}  \frac{ \ell ( d- 1)^2 }{  \lambda \psi'' (1) b   } e^{ - \frac{ \pi b }{ \omega \kappa} }  \frac{( 1 - \kappa^2) \omega }{ \kappa}( 1 + o (1)).
$$
By construction, with this choice of $a$, we have $(1 - \lambda) y'(\tau) = \varphi ( y (\tau))$.  Now, in the domain $1 - \si  \leq y \leq 1$ the non-linear differential equation $y'' - (1 - \lambda) y ' + \tilde \varphi (y)$
coincides with \eqref{eq:odey20}.  Thus,  using  \eqref{eq:phiUB0} and Lemma \ref{le:compode0}$(i)$ with $\tau = T_+$, we find that 
$$
x'(0)  \leq   y'(0) = -  e^{2}  \frac{ \ell ( d- 1)^2 }{ \lambda \psi''(1) b   } e^{ - \frac{ \pi b }{ \omega \kappa} } ( 1 - \kappa^2) \omega^2   (  1+ o ( 1)).
$$
We finally take $\kappa = 1 -  \omega / ( \pi b ) $ and use \eqref{eq:elllambda0}. It proves the lemma.
\qed\end{proof}

\paragraph{Step three : end of proof.}

We may now complete the proof of Theorem \ref{th:RSext}. We start with the left hand side
inequality. We first note that, by  Lemma \ref{le:compode00}, $x'(t)$ is decreasing on the interval $[0,t_0]$ where $t_0$ is the time where $(x(t_0),x'(t_0)) \in \Gamma = \{� (x_1, x_2 ) \in [\rho, 1]�\times [-1, 0] :  ( 1 - \lambda ) x_2 = \varphi (x_1) \}$.  Moreover by Lemma \ref{le:compode00}, we find $x(t_0) \leq 1 - \delta$. However, by Lemma \eqref{le:Uyprime0}$(ii)$, we have
$$
x(t ) \geq 1 - t.
$$
Hence $ t_0 \geq \delta$. Then, by construction, on the interval $[0,t_0]$, $$x(t) \leq 1 + x'(0) t =1 - |� x'(0) |  t  .$$

Since $\psi (x) \leq x$ on $[\rho, 1]$, it follows from \eqref{eq:qQ0} that the survival probability may be lower bounded as
\begin{eqnarray*}
1 - q(\lambda ) = \int_0 ^\infty ( 1 - \psi( x (t)  ) )  e^{ -t } dt & \geq  & \int_0 ^\infty ( 1 - x (t)  )  e^{ -t } dt  \\
& \geq  &  \int_0 ^{t_0} |� x' (0 )  |� t �  e^{ -t } dt \\
& \geq  & | x' (0) |   \int_0 ^{\delta}  t  e^{ -t } dt.
\end{eqnarray*}
It remains to use Lemma \ref{le:LBx0} and we obtain the left hand side of Theorem \ref{th:RSext}.

We turn to the right hand side inequality. For $X = (x_1, x_2) \in [\rho,1] \times (-\infty,0)$, define $G(X) = ( x_2, (1 - \lambda ) x_2 )$.  From
the definition of  $F$ in \eqref{eq:defFX0}, we have,  component-wise, for any $X \in [\rho,1] \times (-\infty,0)$,
$$
F(X) \geq G(X).
$$
Note also that $G$ is monotone : if component-wise $X \geq Y$ then $G(X) \geq G(Y)$. It follows that
if $X(0 ) = Y(0)$, $X' = F(X)$ and $Y' = G (Y)$ then component-wise
$$
X(t) \geq Y(t),
$$
(see e.g.  \cite[Exercise 4.6]{MR0171038}). Looking at the solution of $y'' - ( 1 - \lambda) y' = 0$ such that $y(0) = 1$ and $y'(0 ) = x'(0)$, we get that
$$
x(t) \geq 1 + x'(0) (e^{(1 - \lambda) t} - 1).
$$
We deduce from \eqref{eq:qQ0}-\eqref{eq:defx0} and the convexity of $\psi$ that,
\begin{align*}
q(\lambda)  & = \int_0 ^\infty \psi (  x (t)  ) e^{ - t } dt \\
& \geq \int_0 ^\infty \psi (  1 + x'(0) (e^{(1 - \lambda) t} - 1) ) e^{ -t } dt \\
& \geq  \int_0 ^\infty  \PAR{ 1 +  d x'(0) (e^{(1 - \lambda) t} - 1) } e^{ -t } dt \\
& \geq 1  + d x'(0) / \lambda.
\end{align*}
We finally apply Lemma \ref{le:UBx0} and this concludes the proof of Theorem \ref{th:RSext}. \qed

\section{Proofs of Theorems \ref{th:RStail}-\ref{th:RStail2}}
 \label{sec:RStail}

\subsection{Proof of Theorem \ref{th:RStail2}} \label{subsec:N1}

We are first going to find a recursive distributional equation (RDE) associated to the total progeny of the chase-escape process on a Galton-Watson tree. 

As already pointed, we can build the chase escape process on the tree $T^\downarrow$ thanks to i.i.d. $\EXP(\lambda)$ variables $(\xi_v)_{v \in V}$ and independent  i.i.d. $\EXP(1)$ variables $(D_v)_{v \in V}$. The variable $\xi_v$ (resp. $D_v$) is the time by which $v \in V$ will be infected (resp. recovered) once its ancestor is infected  (resp. recovered).  For $t \geq 0$, we define $Y(t)$ as the total number of recovered vertices when the process reach its absorbing state (without counting $o$) when we replace $D_{\so}$ by  $t$. The variable $Y(t)$ is the conditional variable $Z$ conditioned on the root is recovered at time $t$. By definition, if $D$ is an independent exponential variable with mean $1$,
then $$Z \stackrel{d}{=} Y(D),$$ where the
symbol $\stackrel{d}{=}$ stands for distributional equality.

In $T$, we denote the offsprings of the root by $\{1,\cdots,N\}$. The random variable  $N$ has distribution $P$. The root infects each of its offspring after an independent exponential variable with intensity $\lambda$.  Note that in $T$, the subtrees generated by each of the offsprings of the root are iid copies of $T$.  Hence, the recursive
structure of the tree $T$ leads to the
following equality in distribution
\begin{eqnarray}
\label{eq:RDE2}
Y(t) & \stackrel{d}{=}  & 1 +  \sum_{i = 1}^N \ind ( \xi_i  \leq  t)  Y_i ( t - \xi_i + D_i).
\end{eqnarray}
where $(\xi_i)_{i \in \dN}$ are iid exponential variables with intensity $\lambda$, $(Y_i)_{1 \leq i \leq N}$ and $(D_i)_{1 \leq i \leq N}$  are independent copies of $Y$ and $D$ respectively. Note that since all variables are non-negative, there is
no issue with the case $Y(t) = + \infty$ in the above RDE.

The RDE (\ref{eq:RDE2}) is the cornerstone of the argument. In the remainder of this subsection, we will use it to derive a linear second order ODE for the first moment of $Y(t)$. In the following subsection \S \ref{subsec:rec}, we will extend this exact computation to any integer moment. Finally, using convexity inequalities, we will push further the method and obtain sharp lower and upper bounds for any moment of $Y(t)$ (in \S \ref{subsec:RStailLB} and \S \ref{subsec:RStailUB}).

We start with a lemma
\begin{lemma}
\label{le:YNfinite} Let $t > 0$ and $u \geq 1$, if $\dE'_\lambda [ Z^u]<
\infty$ then $\dE' _\lambda [ Y(t)^u ]< \infty$.
\end{lemma}
\begin{proof}
Since $Z  \stackrel{d}{=}  Y(D)$, from Fubini's Theorem, $ \dE'_\lambda [�Z^u ]  =
\int_{0}^{\infty} \dE'_\lambda [ Y (t) ^u  ]e^{-t} dt$. Therefore $\dE'_\lambda [�Y (t) ^u ]� < \infty$ for almost all $t \geq 0$. Note however that since $t
\mapsto Y(t)$ is monotone for the stochastic domination, it implies that $\dE'_\lambda [ Y(t) ^u ]� < \infty$ for all $t \geq 0$. \qed\end{proof}

Now, assume that $\dE'_\lambda Z < \infty$. We may then take expectation in (\ref{eq:RDE2}):
\begin{eqnarray*}
\dE'_\lambda Y(t) & = & 1 +d \int_0 ^t \int_0^\infty  \dE'_\lambda [ Y (t - x + s)  ] e^{ -s } ds \lambda e^{ - \lambda x} dx.
\end{eqnarray*}
Let $f_1(t) = \dE'_\lambda Y(t)$, it satisfies the integral equation, for all $t \geq 0$,
\begin{equation}
\label{eq:g}
f_1(t) = 1+    \lambda d e^{-\lambda t} \int_0 ^t e^{(\lambda + 1) x} \int_x^\infty  f_1 (s) e^{ -s } ds  dx.
\end{equation}
Multiplying by $e^{\lambda t}$ and taking derivative, we get:
$$
(f'_1 (t) + \lambda f_1 (t)  ) e^{ \lambda t} = \lambda e^{\lambda t} +   \lambda d e^{(\lambda + 1) t} \int_t^\infty  f_1 (s) e^{ -s } ds.
$$
Then, multiplying by $e^{-(\lambda +1) t}$, taking the derivative a second time and  then re-multiplying  by $e^{t}$, we obtain:
$
f_1'' (t) - ( 1 - \lambda)  f_1' (t) - \lambda f_1  =  - \lambda  - \lambda d  f_1(t).
$
So, finally, $f_1$ solves a linear ordinary differential equation of the second order
\begin{eqnarray*}
x'' - (1- \lambda) x' + \lambda  (d- 1) x = -\lambda,
\end{eqnarray*}
with initial condition $f_1(0) = 1$. We get that $$f_1 (t) = x(t) - \frac {1}{d-1},$$ where $x(t)$ solves the ordinary differential equation
\begin{equation}
\label{eq:ODE1}
x'' - (1- \lambda) x' + \lambda  (d- 1) x = 0.
\end{equation} with $x(0) = d / (d-1)$. The discriminant of the polynomial $X^2 - (1- \lambda) X + \lambda  (d- 1)  = 0$ is $$ \Delta = \lambda^2 - 2 \lambda ( 2 d - 1) + 1.$$

If $0 < \lambda < \lambda_1$, the discriminant is positive. The roots of the polynomial are
\begin{equation}
\label{eq:alphabeta}
 \alpha = \frac{ 1 - \lambda - \sqrt \Delta}{ 2}   \quad  \hbox{ and } \quad \beta = \frac{ 1 - \lambda+ \sqrt \Delta}{  2}.
\end{equation}
The solutions of (\ref{eq:ODE1}) are
\begin{equation}\label{eq;xdefa}
x ( t )  =  \frac{ d }{d-1}  \left(  (1 - a) e^{\alpha t } + a e^{\beta t}\right)
\end{equation}
for some constant $a$.

Similarly, if $\lambda = \lambda_1$, then $\alpha = \beta = 1- d +   \sqrt {�d( d - 1) }$ and the solutions of  (\ref{eq:ODE1}) are
\begin{equation}\label{eq;xdefa2}
x ( t )  = \frac{ d }{d-1}  (at +1 ) e^{ \alpha t  }.
\end{equation}
For $0 < \lambda \leq \lambda_1$, we check easily that the functions $x(\cdot)$ with $a \geq 0$ are the nonnegative solutions of the integral equation (\ref{eq:g}).

It remains to prove that if  $0 < \lambda \leq  \lambda_1$ then $\dE Z < \infty$ and
$$
f_{1}  ( t ) =  \frac{  d e^{\alpha t }  - 1 }{d-1}  .
$$
Indeed, we would get  $\dE Z = \int_0^ {\infty} \dE Y(t) e^{-t} dt =   \frac {d}{d-1} \frac{1}{1 - \alpha} - \frac{1}{d-1} $ as stated in Theorem \ref{th:RStail2}. To this end, define $T_n$ as the tree $T$ stopped at generation $n$. As above, we denote by $Y^{(n)} (t)$ the total number of recovered particles in $T_n$ when the root is recovered at time $t$ ($D_{\so}$ is replaced by $t$). As $n \to \infty$, $Y_n(t)$ is non-decreasing and converges to $Y(t)$. We have $Y^{(0)}(t) = 1$ and for all $n \geq 0$, as in RDE (\ref{eq:RDE2}),
\begin{eqnarray*}
\label{eq:RDE2n}
Y^{(n+1)}(t) & \stackrel{d}{=}  & 1 +  \sum_{i = 1}^N \ind ( \xi_i  \leq  t)  Y^{(n)}_i ( t - \xi_i + D_i) ,
\end{eqnarray*}
where  $Y^{(n)}_i,$ and $D_i$ are independent copies of $Y^{(n)}$ and $D$ respectively. Since $\dE N < \infty$, the expectation of the number of vertices in $T_n$ is finite. In particular, for all $n \geq 0$, $g_n  (t) = \dE Y^{(n)}(t) < \infty$ is bounded uniformly in $t$. Also, taking expectation in \eqref{eq:RDE2n}, we have for all $t \geq0$,
\begin{equation}
\label{eq:gn}
g_{n+1}(t) = 1+    \lambda d e^{-\lambda t} \int_0 ^t e^{(\lambda + 1) x} \int_x^\infty  g_n (s) e^{ -s } ds  dx = \Phi ( g_n) (t),
\end{equation}
where  $\Phi$ is the mapping
$$
\Phi : g \mapsto  1 + \lambda d e^{-\lambda t} \int_0 ^t e^{(\lambda + 1) x} \int_x^\infty  g (s) e^{ -s } ds  dx.
$$
It is easy to check that $\Phi$ is indeed a mapping from $\cH_1$ to $\cH_1$, where $\cH_1$ is the set of non-decreasing functions
$g : [0,\infty) \to [1,\infty)$ such that $\sup_{t \geq 0} g(t) e^{-\alpha t} < \infty$. Now, from what precedes the function
$$
h (t) = \frac{  d e^{\alpha t }  - 1 }{d-1}  .
$$
is a fixed point of $\Phi$.

Denote by $\leq$ the partial order on $\cH_1$ of point-wise domination: $g \leq f$ if for all $t \geq 0$, $g(t) \leq f(t)$. The mapping $\Phi$ is non-decreasing on $\cH_1$ for this partial order.  We notice that $g_0\leq h$ and $ g_0 \leq g_1$. Composing by $\Phi$, we obtain: $g_1 = \Phi  (g_0 ) \leq \Phi  (h) = h$ and $g_1 \leq g_2$. By recursion, it follows for any $n \geq 1$, $g_n \leq h$ and $(g_n)_{n \geq 0}$ is non-decreasing. By monotone convergence, for any $t \geq 0$, the limit $g (t) =  \lim_{n \to \infty}
g_n(t)$ exists and is bounded by $h(t)$. Also, since $ g_n  \leq h$, by dominated convergence, for any $t \geq 0$, $\lim_{n \to \infty} \Phi ( g_n ) (t) = \Phi ( g ) (t) $. Therefore $g$ solves the integral equation (\ref{eq:g}) and is
equal to $x - 1/(d-1)$ where $x$ is given by \eqref{eq;xdefa} (or \eqref{eq;xdefa2} if $\lambda = \lambda_1$) for some $a \geq 0$.  However, from what precedes, we get $x (t) \leq h (t) + 1/(d-1)$ and the only possibility is $a =0$ and $g (t) = h(t)$.

Finally, since $Y^{(n)}(t)$ is non-decreasing and converges to $Y(t)$, by monotone convergence we have that $f_1(t) = \dE Y(t) = \lim_{n \to \infty} \dE Y^{(n)}(t) = g(t)$. This concludes the proof of Theorem \ref{th:RStail2}. \qed

\subsection{Proof of Theorem \ref{th:RStail} for integer moments}

\label{subsec:rec}

For $0 < \lambda < \lambda_1$, we define
\begin{equation}
\label{eq:gamma} \overline \gamma =   \frac{ \lambda^2 - 2 d
\lambda +1 - (1 - \lambda) \sqrt{ \Delta }}{2\lambda (d -1)}  = \frac  \beta   \alpha,
\end{equation}
where $\alpha$ and $\beta$ are given by \eqref{eq:alphabeta}. The key property of $\og (\lambda) $ is that $(1 - \lambda)  u \alpha - \lambda (d-1)- u ^2 \alpha^2 >0$ if and only if $1 < u <\og (\lambda)$. We also note that if $u > 1$, $ u < \og$ is equivalent to
$\lambda \in (0,\lambda_u)$. We first state an important lemma. Let  $1 < u < \og$,  we define $\cH_u$, the set of measurable functions
$h : [0,\infty) \to [0,\infty)$ such that $\sup_{t \geq 0} h(t) e^{-u\alpha t} < \infty$. Let $L> 0$, we
define the mapping from $\cH_u$ to $\cH_u$,
$$
\Psi : h \mapsto   L e^{u \alpha t }   +    \lambda d e^{-\lambda t} \int_0 ^t e^{(\lambda + 1) x} \int_x^\infty  h (s) e^{ -s } ds  dx.
$$
In order to check that $\Psi$ is indeed a mapping from $\cH_u$ to
$\cH_u$, we use the fact that if $1 < u <\overline \gamma = \beta / \alpha $ then
$u \alpha < \beta < 1$.
\begin{lemma}
\label{le:powerupp2}
Let  $1 < u < \overline \gamma$ and $f \in \cH_u$ such that $f \leq \Psi(f)$. Then for all $t \geq 0$,
$$
f(t) \leq   \frac{   L (u \alpha  + \lambda)( 1  - u \alpha)    }{(1 - \lambda)  u\alpha - \lambda (d-1)- u^2 \alpha^2}  \, e ^{u \alpha t} .
$$
\end{lemma}
\begin{proof} We set $g_0 = f$ and  for $k \geq
1$, we define $g_{k} = \Psi (g_{k-1})$. First, since $1 < u < \overline \gamma$ then $(u \alpha  + \lambda)( 1  - u \alpha) > \lambda d$. We use the formula for all $u \geq 0$ such that  $  u \alpha< 1$:
\begin{equation}\label{eq:defjefio}
 \lambda  e^{-\lambda t} \int_0 ^t e^{(\lambda + 1) x} \int_x^\infty  e^{ \alpha  u  s}  e^{ -s } ds  dx = \frac{ \lambda ( e^{ \alpha u t } - e^{- \lambda t } ) }{(   u\alpha + \lambda )   ( 1 -  u \alpha ) } .
\end{equation}
We deduce easily that if $ g_0 (t) \leq C_0 e^{u \alpha t} $ then $$g_1
(t) = \Psi (g_0) (t)  \leq  L e^{u \alpha t } +  \frac{ L \lambda d }{(u \alpha  + \lambda)( 1  - u \alpha) }  (e^{u \alpha t } -e^{-\lambda t}) \leq C_1  e^{u
\alpha t}, $$ with $C_1 =  L + \frac{ C_0 \lambda d }{(u \alpha  + \lambda)( 1  - u \alpha) }$. By recursion, we obtain that $\limsup_k g_k (t) \leq
C e^{u\alpha t}$, with $C =  L (u \alpha  + \lambda)( 1  - u \alpha) / ( (1 - \lambda)  u\alpha - \lambda (d-1)- u^2 \alpha^2 )  < \infty$.

We may now conclude the proof. Notice that $\Psi$ is monotone: if for all $t \geq 0$, $h_1 (t) \geq  h_2 (t)$ then for all $t \geq 0$, $\Psi (h_1) (t) \geq \Psi (h_2) (t)$. Hence, by recursion, from the assumption $f \leq \Psi(f) = g_1$, we deduce that for all integer $k \geq 1$,  $f \leq g_k$. It remains to take the limit in $k$. \qed\end{proof}

Now, let $p$ be an integer, and define $f_p (t) = \dE'_\lambda [ Y(t)^p ]$. The main result of this subsection is the following lemma.

\begin{lemma}
\label{le:rec}
Let $1 \leq p < \gamma_{P}$, if $\lambda \in (0,\lambda_p)$, then $f_p$ is finite and there exists a constant $C_p$ such that for all $t >0$ $$ f_p (t) \leq C_p  e^{p \alpha t}.$$
\end{lemma}
\begin{proof}
In \S \ref{subsec:N1}, we have computed $f_p$ for $p=1$ and found $f_1 (t) = (d e^{\alpha t} -1)/(d-1)$. Let $p \geq  2$ and assume now that the statement of Lemma \ref{le:rec} holds for $q =1, \cdots, p-1$. Let $\kappa > 0$, $Y^{(\kappa)}(t) = \min( Y(t), \kappa)$ and let $\leq_{st}$ denote the stochastic domination (beware that $Y^{(\kappa)} (t)$ is different from $Y^{(n)}(t)$ defined in  \S \ref{subsec:N1}). We use that if $y_i \geq 0$, $\min( \sum_i y_i  , \kappa ) \leq  \sum_i \min( y_i  , \kappa ) $. Hence, from RDE (\ref{eq:RDE2}), we have
 \begin{equation} \label{eq:ineqRDE0}
Y^{(\kappa)}(t)  \leq_{st}     1  +  \sum_{i=1} ^N   \ind ( \xi_i  \leq  t)  Y^{(\kappa)}_i ( t - \xi_i + D_i).
 \end{equation}
Recall the multinomial formula
$$
\left( \sum_{i=1} ^ n y_i \right)^p = \sum_{p_1, \cdots, p_n} {{n}\choose{p_1 \cdots p_n}} y_1^{p_1} \cdots y_n ^{p_n} .
$$
where the summation is taken over $n$-tuples of integers that sum up to $p$.
Taking power $p$ in the above stochastic inequality and expanding brutally, we thus get
\begin{equation*} \label{eq:ineqRDE}
Y^{(\kappa)}(t)  ^p  \leq_{st}     \sum_{p_1, \cdots, p_{N+1}} { {N+1}\choose{p_1 \cdots p_{N+1} } } \prod_{i=1}^N   \left( \ind_{p_i = 0} + \ind_{p_i \geq 1}  \ind ( \xi_i \leq  t)  Y^{(\kappa)}_i ( t - \xi_i + D_i) ^{p_i}  \right),
\end{equation*}
where the summation is taken over $N+1$-tuples of integers that sum up to $p$. Now we define $$f_p ^{(\kappa)} (t) = \dE'_\lambda \left[ Y^{(\kappa)} (t)^p \right] = \dE'_\lambda \left[  \min (Y (t) ,  \kappa )^p\right] .$$
Taking expectation and using independence leads to
\begin{eqnarray}
f^{(\kappa)}_p (t)   & \leq  & \sum_{n=0}^\infty P(n) \sum_{p_1, \cdots, p_{n+1}} { {n+1}\choose{p_1 \cdots p_{n+1} } }     \prod_{i=1}^n   \left( \ind_{p_i = 0} + \ind_{p_i \geq 1} \dE'_\lambda  \left[ \ind ( \xi  \leq  t)   Y^{(\kappa)}( t - \xi + D) ^{p_i} \right] \right)  \nonumber \\
 & \leq &  \sum_{n=0}^\infty P(n) \sum_{p_1, \cdots, p_{n+1}} { {n+1}\choose{p_1 \cdots p_{n+1} } }  \nonumber  \\
 & & \hspace{65pt} \times  \prod_{i=1}^n   \left( \ind_{p_i = 0} + \ind_{p_i \geq 1}  \lambda  e^{-\lambda t} \int_0 ^t e^{(\lambda + 1) x} \int_x^\infty  f^{(\kappa)}_{p_i} (s) e^{ -s } ds  dx \right). \label{eq:fprec0}
\end{eqnarray}
Consider a $n+1$-tuple that sums up to $p$ such that for all $i = 1 , \cdots, n$, $p_i \leq p-1$, $\sum_{i=1} ^{n} p_i = q \leq p$ and $\sum_{i=1} ^ n   \ind_{p_i \geq 1} = m \leq p$.  From the recursive hypothesis and \eqref{eq:defjefio}, with $L = \max_{1 \leq k \leq p-1} C_k$, we get
\begin{eqnarray*}
 \prod_{i=1}^n   \left( \ind_{p_i = 0} + \ind_{p_i \geq 1}  \lambda  e^{-\lambda t} \int_0 ^t e^{(\lambda + 1) x} \int_x^\infty  f^{(\kappa)}_{p_i} (s) e^{ -s } ds  dx \right) & \leq &  \prod_{i : p_i \geq 1} \frac{C_{p_i}  \lambda e^{ \alpha p_i t }  }{( \lambda +  p_i  \alpha )   ( 1 - p_i \alpha ) } \\
 & \leq &   L^m e^{ \alpha q t } e^{ - \sum_{i =1}^n  \ln ( 1 - p_i \alpha ) }.
 \end{eqnarray*}
Now recall that $| \ln ( 1 - y ) + y | \leq \frac {y^2}{2 (1- y)}$ for $y \in (0,1)$. Since $\sum_{i=1}^n p_i ^2 \leq q^2$, we get
\begin{eqnarray*}
\prod_{i=1}^n   \left( \ind_{p_i = 0} + \ind_{p_i \geq 1}  \lambda  e^{-\lambda t} \int_0 ^t e^{(\lambda + 1) x} \int_x^\infty  f_{p_i} (s) e^{ -s } ds  dx \right) & \leq &   L^m  e^{ \alpha q t } e^{ \alpha q + \frac{ \alpha^2 q^2}{ 2 (1 - p \alpha) } }\\
& \leq & L' e^{ \alpha p t }.
 \end{eqnarray*}
 Then, grouping together all such $n+1$-tuples, from (\ref{eq:fprec0}) we deduce
 \begin{eqnarray}
 f^{(\kappa)}_p (t)  & \leq &   \sum_{n=0}^\infty P(n) \left((n+1)^p  L' e^{ \alpha p t }  + n \lambda  e^{-\lambda t} \int_0 ^t e^{(\lambda + 1) x} \int_x^\infty  f^{(\kappa)}_{p} (s) e^{ -s } ds  dx \right)  \nonumber \\
 & \leq &   L'' e^{ \alpha p t }  + \lambda d e^{-\lambda t} \int_0 ^t e^{(\lambda + 1) x} \int_x^\infty  f^{(\kappa)}_{p} (s) e^{ -s } ds  dx , \label{eq:fprec}
 \end{eqnarray}
where  we have used the hypothesis $p < \gamma_{P}$.   We may then apply Lemma \ref{le:powerupp2}: there exists a constant $C_{p}$ such that
$$ f_1^{(\kappa)} (t)\leq  C_p e^{ \alpha p t}.$$
 The monotone convergence Theorem implies that $f_p (t) = \lim_{\kappa \to \infty}
f_p^{\kappa)} (t)$ exists and is bounded by $C_p e^{ \alpha p t}$. The recursion is complete.
\qed\end{proof}

\subsection{Proof of Theorem \ref{th:RStail} : lower bound on $\gamma (\lambda)$}
\label{subsec:RStailLB}

To prove Theorem \ref{th:RStail}, we shall prove two statements

\begin{equation}
\label{eq:power1}
\hbox{If $ \dE'_\lambda [ Z^u]  < \infty$  then }  u \leq \overline \gamma,
\end{equation}

\begin{equation}
\label{eq:power2}
\hbox{If $1 < u < \min(\overline \gamma,\gamma_{P})$ then } \dE'_\lambda [ Z^u]  < \infty.
\end{equation}

In this paragraph, we prove (\ref{eq:power2}). The argument is a refinement of the argument in \S \ref{subsec:rec}. Let $\kappa > 0$ and let $f_u ^{(\kappa)} (t) =\dE'_\lambda [\min (Y(t),\kappa)^u]$, we have the following lemma.
\begin{lemma}
\label{le:powerupp}
If $1 < u < \min(\overline \gamma,\gamma_{P})$, there exists a constant $L_u > 0$ such that for all $t \geq 0$ and $\kappa > 0$,
$$
f_u ^{(\kappa)} (t)   \leq   L_u  e^{u \alpha t }   + \lambda d e^{-\lambda t} \int_0 ^t e^{(\lambda + 1) x} \int_x^\infty  f^{(\kappa)}_{u} (s) e^{ -s } ds  dx.
$$
\end{lemma}
\begin{proof}
The lemma is already proved if $u$ is an integer in (\ref{eq:fprec}). The general case extends of the same argument. We write $u = p  + v$ with $v \in (0,1)$ and integer $p \geq 1$. We use the inequality, for all $y_i \geq 0$, $1 \leq i \leq n$,
\begin{equation}\label{eq:multiu}
\left( \sum_{i=1} ^n y_i  \right) ^ u \leq  \sum_{ i = 1} ^n  \sum_{p_1, \cdots, p_n} {{n}\choose{p_1 \cdots p_n}}  y_i ^ {p_i + v} \prod_{1 \leq  j  \leq n ,  j  \neq i } y_j^{p_j},
\end{equation}
where the summation is taken over $n$-tuples of integers that sum up to $p$ (which follows from the inequality $(\sum y_i) ^{v} \leq \sum y_i^{v}$ and the multinomial formula). Then from (\ref{eq:ineqRDE0}),  we get the stochastic domination
\begin{eqnarray}
Y^{(\kappa)}(t)  ^u  & \leq_{st}  &    \sum_{ i = 1} ^N \sum_{p_1, \cdots, p_{N+1}} { {N+1}\choose{p_1 \cdots p_{N+1} } }  \left( \ind_{p_i = 0} + \ind_{p_i \geq 1}  \ind ( \xi_i \leq  t)  Y^{(\kappa)}_i ( t - \xi_i + D_i) ^{p_i+ v}  \right) \nonumber \\
& & \hspace{50pt} \times  \prod_{1 \leq  j  \leq N ,  j  \neq i }   \left( \ind_{p_j = 0} + \ind_{p_j \geq 1}  \ind ( \xi_j \leq  t)  Y^{(\kappa)}_j ( t - \xi_j + D_i) ^{p_j}  \right), \label{eq:ineqRDEu}
\end{eqnarray}
where the summation is taken over $N+1$-tuples of integers that sum up to $p$. From Lemma \ref{le:rec}, for all $1 \leq q \leq p $, $f_q (t) \leq C_q e^{q \alpha t}$. Note also, by Jensen inequality, that for all $1 \leq q \leq p-1$,
$f_{q+v} (t) \leq  f_{p} (t) ^{\frac{q+v}{p}} \leq C_p e
^{(q+v)\alpha t}$.
The same argument (with $p$ replaced by $u$)
which led to (\ref{eq:fprec}) in the proof of Lemma
\ref{le:rec} leads to the result. \qed\end{proof}

Statement (\ref{eq:power2}) is a consequence of Lemma  \ref{le:powerupp2} and Lemma \ref{le:powerupp}. Indeed, by Lemma \ref{le:powerupp2}, for all $t \geq 0$, $f_u ^{(\kappa)} (t)\leq C_u e^{u \alpha t}$ for some positive constant $C_u$ independent of $\kappa$. From the monotone convergence Theorem, we deduce that, for all $t \geq 0$, $f_u (t) \leq C_u e^{u \alpha t}$. However from $Z \stackrel{d}{=} Y(D)$, we find
$$
\dE'_\lambda Z^u = \int_0 ^\infty f_u (t) e^{ - t} dt \leq \int_0 ^\infty   C_u e^{u \alpha t} e^{ - t}dt.
$$
Then, statement (\ref{eq:power2}) follows from  $u \alpha < 1$.

\subsection{Proof of Theorem \ref{th:RStail} : upper bound on $\gamma (\lambda)$}

\label{subsec:RStailUB}

In this paragraph, we prove statement (\ref{eq:power1}). This will conclude the proof of Theorem  \ref{th:RStail}.  Let $u > 1$, we assume that $ \dE'_\lambda[  Z^u]� < \infty$ we need to show that $\lambda < \lambda_u$. Without loss of generality we can assume that $\lambda < \lambda_1$. From Lemma \ref{le:YNfinite} and (\ref{eq:RDE2}), we get
 \begin{equation*}
\label{eq:RDEp}
f_u (t) = \dE'_\lambda [ Y(t)^u  ] = \dE'_\lambda \left(  1 +  \sum_{i = 1}^N \ind ( \xi_i  \leq  t)  Y_i ( t - \xi_i + D_i) \right)^u .
\end{equation*}
Taking expectation and using the inequality $(x + y)^u \geq x^u + y^u$, for all positive $x$ and $y$, we get:
\begin{eqnarray}
f_u (t)  &  \geq & 1 +   \lambda d e^{-\lambda t} \int_0 ^t e^{(\lambda + 1) x} \int_x^\infty  f_{u} (s) e^{ -s } ds  dx.  \label{eq:gplow}
\end{eqnarray}
From Jensen's Inequality, $f_u (t) \geq f_1 (t) ^ u \geq e^{u \alpha t}$. Note that the integral $\int_x ^ {\infty} e^{\alpha u s}  e^{-s} ds$ is finite if and only if  $u <   \alpha^{-1}$.  Suppose now that $\overline \gamma < u < \alpha^{-1}$.  We use the fact that if  $ u > \overline \gamma$ then   $u^2 \alpha^2  - (1 - \lambda)  u\alpha + \lambda (d-1) > 0$. It implies that there exists $0 < \epsilon < \lambda$ such that
\begin{equation}
\label{eq:defeps}
u^2 \alpha^2  - (1 - \lambda)  u\alpha + \lambda (d-1)  >   \epsilon d .
\end{equation}
We define $\tilde \lambda = \lambda - \epsilon$, $\tilde \Delta (\epsilon) = (1 - \lambda)^2 - 4 (\tilde \lambda d - \lambda)$. Note that $\tilde \Delta (0) =\Delta$.  Since $\lambda < \lambda_1$, for $\epsilon$ small enough, $\tilde \Delta$ is non-negative, we may then consider  the real roots of $X^2 - (1 - \lambda) X + \tilde \lambda d - \lambda$:
$$ \tilde \alpha  (\epsilon)= \frac { 1 - \lambda - \sqrt { \tilde\Delta} }{ 2} \quad \hbox{ and } \quad \tilde \beta (\epsilon) =\frac { 1 - \lambda + \sqrt{ \tilde\Delta} }{ 2}.$$
Again, for $\epsilon =0$, $\tilde \alpha (0) = \alpha$ and $\tilde \beta (0) = \beta$. Hence, since $ u > \og =  \beta / \alpha$, by continuity, for $\epsilon$ small enough,
\begin{equation}
\label{eq:deftilde2}
u \alpha > \tilde \beta.
\end{equation}
We compute a lower bound from (\ref{eq:gplow}) as follows:
\begin{eqnarray*}
f_u (t)  &  \geq &1 + \epsilon d e^{-\lambda t} \int_0 ^t e^{(\lambda + 1) x} \int_x^\infty  f_{u} (s) e^{ -s } ds  dx +   \tilde \lambda d e^{-\lambda t} \int_0 ^t e^{(\lambda + 1) x} \int_x^\infty  f_{u} (s) e^{ -s } ds  dx \nonumber \\
& \geq & 1 + \epsilon d e^{-\lambda t} \int_0 ^t e^{(\lambda + 1) x} \int_x^\infty  e^{ u \alpha s} e^{ -s } ds  dx +   \tilde \lambda d e^{-\lambda t} \int_0 ^t e^{(\lambda + 1) x} \int_x^\infty  f_{u} (s) e^{ -s } ds  dx   \nonumber \\
& \geq & 1 + L  ( e^{u \alpha t } - e^{-\lambda t} )  +  \tilde \lambda d e^{-\lambda t} \int_0 ^t e^{(\lambda + 1) x} \int_x^\infty  f_{u} (s) e^{ -s } ds  dx  \label{eq:gplow2},
\end{eqnarray*}
with $$L =  \frac{ \epsilon  d }{ ( \lambda + u\alpha )(1 - u\alpha) } >0.$$
We consider the mapping $\Psi : h \mapsto 1 + L  ( e^{u \alpha t } - e^{-\lambda t} )  + \tilde \lambda \int_0 ^ t e^x \int_x ^ {\infty} f_u (s) e^{-s} ds dx$. $\Psi$ is monotone: if for all $t \geq 0$, $h_1 (t) \geq  h_2 (t)$ then for all $t \geq 0$, $\Psi (h_1) (t) \geq \Psi (h_2) (t)$.   Since, for all $t \geq 0$,  $f_u(t)  \geq \Psi (f_u)(t) \geq 1$, we deduce by iteration that there exists a function $h$ such that $h = \Psi (h) \geq 1$.  As in \S \ref{subsec:N1}, solving $h = \Psi (h)$ is simple, taking twice the derivative, we get,
$$
h'' - (1 - \lambda) h' + (\tilde \lambda d - \lambda)  h = - \lambda - L  ( \lambda + u\alpha )(1 - u\alpha) e^{u \alpha t} .
$$
Therefore, $h = a e^{\tilde \alpha t } + b e^{\tilde \beta t } - \epsilon ( u^2 \alpha^2  - (1 - \lambda)  u\alpha + \lambda (d-1) ) ^{-1}e^{u \alpha t }  $ for some constant $a$ and $b$. From (\ref{eq:deftilde2}) the leading term as $t$ goes to infinity is equal to $- \epsilon ( u^2 \alpha^2  - (1 - \lambda)  u\alpha + \lambda (d-1) ) ^{-1}e^{u \alpha t } $. However from (\ref{eq:defeps}), $- \epsilon ( u^2 \alpha^2  - (1 - \lambda)  u\alpha + \lambda (d-1) ) ^{-1}  < 0$ and it contradicts the assumption that $h (t) \geq 1$ for all $t \geq 0$. Therefore we have proved that $u \leq \overline \gamma$. \qed

\section*{Appendix}

In this appendix, for the sake of completeness we include the proof of the following lemma on Galton-Watson trees.

\begin{lemma}\label{le:daryGWT}
Let $T$ be a Galton-Watson tree with mean number of offsprings $d  > 1$. Conditioned on $T$ is infinite, $T$ is a.s. lower $d$-ary.
\end{lemma}

\begin{proof}
Let $1 < \delta < d$ and $Z_n = |V_n| $ be the number of offsprings of generation $n$. From Seneta-Heyde Theorem (see \cite[Chapter 5]{lyonsbook}), conditioned on $T$ is infinite, a.s. 
$$\lim_{n \to \infty} \frac 1 n \log Z_n  = d.$$
Let $p > 0$ be the probability that $T$ is infinite. It implies that for any $ \veps > 0$, for all $n$ large enough, we have  $\dP ( Z_n \geq \delta ^ n ) \geq p - \veps$.

Now, consider a new Galton-Watson tree $T'$ starting from the root $\o$ where each vertex produces independently $m = \lfloor \delta ^ n \rfloor$ offsprings with probability $p - \veps$ and $0$ offspring otherwise. From what precedes, we may couple $T$ and $T'$ such that $T'$ is a subtree of $T^{*n}$.

We are now going to prove that $T'$ contains a large regular tree with probability at least $p - 2 \veps$. To this end, we set $$q = q(\veps) = 1 -  p + \veps.$$
Note that we may have chosen $n = n(\veps)$ large enough so that
\begin{equation}\label{eq:defqnveps}
q + ( 1 - q) e^{ -  m \veps^2 /2}  \leq q + \veps.
\end{equation}

We consider the following pruning algorithm on $T'$. At step $0$, we start with all vertices of $T'$. At step $1$, we remove all vertices which have less than $( 1 - q - 2 \veps ) m$ offsprings. We now iterate: at step $k \geq 1$, we remove all vertices which have less than $( 1 - q - 2 \veps ) m$ offsprings left by step $k-1$.

Denote by $\rho_k$ the probability that the root of $T'$ is removed by step $k$. We have $\rho_0 = 0$, $\rho_1 = q$ and $(\rho_k)_{k \geq 0}$ is an non-decreasing sequence. We are going to check by recursion that for all $k \geq 1$,
\begin{equation}\label{eq:rhokrec}
\rho_k \leq q + \veps.
\end{equation}
Indeed, let $k \geq 1$ and assume that $\rho_{k-1} \leq q + \veps$. Note that if the root is removed by step $k$, then either it has $0$ offspring or more than $( q + 2 \veps ) m$ of its offsprings were removed by step $k-1$. From the recursive structure of the Galton-Watson tree, the probability that an offspring was removed by step $k-1$ is $\rho_{k-1}$ and these events for each offspring are independent. It follows that
\begin{eqnarray*}
\rho_{k} \leq q + (1- q) \dP \PAR{ \sum_{i=1} ^m X_i \geq ( q + 2 \veps) m } ,
\end{eqnarray*}
where $(X_i)_{1 \leq i \leq m}$ are i.i.d. $\{ 0, 1 \}$-Bernoulli variables with mean $\rho_{k-1}$. By recursion hypothesis, $\rho_{k-1} \leq q + \veps$. Hence, Hoeffding's inequality leads to
$$
\dP \PAR{ \sum_{i=1} ^m X_i \geq ( q + 2 \veps) m } \leq \dP \PAR{ \sum_{i=1} ^m ( X_i - \dE X_i )  \geq  \veps m } \leq e^{ -  m \veps^2 /2}.
$$
From \eqref{eq:defqnveps}, we deduce that $\rho_k \leq q + \veps$. This proves \eqref{eq:rhokrec}.

We have thus proven that with probability at least $1 - ( q+\veps)  = p - 2 \veps$, the root of $T$ is never removed by the pruning algorithm. However, on the latter event, by construction $T'$ contains a $\lfloor ( 1 - q - 2 \veps ) m \rfloor $-ary tree rooted at $\o$ (note that $1 -  q - 2 \veps  = p - 3 \veps$).

We may now conclude the proof. We apply the above argument to some $\delta' \in ( \delta, d)$. This proves that for any $0 < \veps < 1$, there exists an integer $n_\veps$ such that with probability at least $(1 - \veps)p $, $T^{*n_\veps}$ contains a $\lceil \delta^{n_\veps} \rceil$-ary tree. Note that the latter event is contained in the event that $T$ is infinite. It follows that the conditional probability that $T^{*n_\veps}$ contains a $\lceil \delta^{n_\veps} \rceil$-ary tree, given $T$ infinite, is at least $1 - \veps$.

We finally consider the sequence $\veps_k = 1 / k^2$. From Borel-Cantelli lemma, conditioned on $T$ infinite, a.s. there exists $k$ such that $T^{*n_{\veps_k}}$ contains a $\lceil \delta^{n_{\veps_k}} \rceil$-ary tree. \qed \end{proof}



\begin{thebibliography}{10}

\bibitem{MR2834719} Luigi Addario-Berry and Nicolas Broutin, \emph{Total progeny in killed  branching random walk}, Probab. Theory Related Fields \textbf{151} (2011),  no.~1-2, 265--295. \MR{2834719}

\bibitem{MR2737710} Elie A{\"{\i}}d{\'e}kon, \emph{Tail asymptotics for the total progeny of the  critical killed branching random walk}, Electron. Commun. Probab. \textbf{15}  (2010), 522--533. \MR{2737710}

\bibitem{AHZ}
Elie A{\"{\i}}d{\'e}kon, Yueyun Hu, and Olivier Zindy, \emph{The precise tail
  behavior of the total progeny of a killed branching random walk},
  \ARXIV{1102.5536}.

\bibitem{aldouskrebs} David Aldous and William Krebs, \emph{The ``birth-and-assassination'' process},  Statist. Probab. Lett. \textbf{10} (1990), no.~5, 427--430. \MR{1078244}

\bibitem{andersson} H{\aa}kan Andersson, \emph{Epidemic models and social networks}, Math. Sci.  \textbf{24} (1999), no.~2, 128--147. \MR{1746332}

\bibitem{MR2047480} Krishna~B. Athreya and Peter~E. Ney, \emph{Branching processes}, Dover  Publications Inc., Mineola, NY, 2004, Reprint of the 1972 original [Springer,  New York; MR0373040]. \MR{0373040}

\bibitem{MR2774095} Jean B{\'e}rard and Jean-Baptiste Gou{\'e}r{\'e}, \emph{Survival probability of  the branching random walk killed below a linear boundary}, Electron. J.  Probab. \textbf{16} (2011), no. 14, 396--418. \MR{2774095}

\bibitem{MR2453554} Charles Bordenave, \emph{On the birth-and-assassination process, with an  application to scotching a rumor in a network}, Electron. J. Probab.  \textbf{13} (2008), no.~66, 2014--2030. \MR{2453554}

\bibitem{MR1473413} Eric Brunet and Bernard Derrida, \emph{Shift in the velocity of a front due to  a cutoff}, Phys. Rev. E (3) \textbf{56} (1997), no.~3, part A, 2597--2604. \MR{1473413}

\bibitem{MR2571413} Amir Dembo and Ofer Zeitouni, \emph{Large deviations techniques and  applications}, Stochastic Modelling and Applied Probability, vol.~38,  Springer-Verlag, Berlin, 2010, Corrected reprint of the second (1998)  edition. \MR{2571413}

\bibitem{MR2779399} Nina Gantert, Yueyun Hu, and Zhan Shi, \emph{Asymptotics for the survival  probability in a killed branching random walk}, Ann. Inst. Henri Poincar\'e  Probab. Stat. \textbf{47} (2011), no.~1, 111--129. \MR{2779399}

\bibitem{haggstrom98} Olle H{\"a}ggstr{\"o}m and Robin Pemantle, \emph{First passage percolation and  a model for competing spatial growth}, J. Appl. Probab. \textbf{35} (1998),  no.~3, 683--692. \MR{1659548}

\bibitem{MR0171038} Philip Hartman, \emph{Ordinary differential equations}, John Wiley \& Sons  Inc., New York, 1964. \MR{0171038}

\bibitem{MR0254929} Christopher Heyde, \emph{Extension of a result of {S}eneta for the  super-critical {G}alton-{W}atson process}, Ann. Math. Statist. \textbf{41}  (1970), 739--742. \MR{0254929}

\bibitem{kordzakhia} George Kordzakhia, \emph{The escape model on a homogeneous tree}, Electron.  Comm. Probab. \textbf{10} (2005), 113--124 (electronic). \MR{2150700}

\bibitem{kordzakhia05} George Kordzakhia and Steven~P. Lalley, \emph{A two-species competition model  on {$\Bbb Z\sp d$}}, Stochastic Process. Appl. \textbf{115} (2005), no.~5,  781--796. \MR{2132598}

\bibitem{IK}
Igor Kortchemski, \emph{Predator-prey dynamics on infinite trees: a branching
  random walk approach}, \ARXIV{1312.4933}.

\bibitem{MR2793860} Carl Mueller, Leonid Mytnik, and Jeremy Quastel, \emph{Effect of noise on front  propagation in reaction-diffusion equations of {KPP} type}, Invent. Math.  \textbf{184} (2011), no.~2, 405--453. \MR{2793860}

\bibitem{surveySIR} Mark Newman, Albert-L{\'a}szl{\'o} Barab{\'a}si, and Duncan~J. Watts (eds.),  \emph{The structure and dynamics of networks}, Princeton Studies in  Complexity, Princeton University Press, Princeton, NJ, 2006. \MR{2352222}

\bibitem{dynamicinformation}
K.~Ramamritham and P.~Shenoy~(Editors), \emph{Special issue on dynamic
  information dissemination}, IEEE Internet Computing \textbf{11} (2007),
  14--44.

\bibitem{richardson} Daniel Richardson, \emph{Random growth in a tessellation}, Proc. Cambridge  Philos. Soc. \textbf{74} (1973), 515--528. \MR{0329079}

\bibitem{lyonsbook}
{Russell Lyons {\rm with} Yuval Peres}, \emph{Probability on trees and
  networks}, In preparation. Available at {\tt
  http://mypage.iu.edu/\string~rdlyons/}, Cambridge University Press, New York,
  2007.

\bibitem{MR0234530} Eugene Seneta, \emph{On recent theorems concerning the supercritical  {G}alton-{W}atson process.}, Ann. Math. Statist. \textbf{39} (1968),  2098--2102. \MR{0234530}

\bibitem{tsitsiklis} John Tsitsiklis, Christos Papadimitriou, and Pierre Humblet, \emph{The  performance of a precedence-based queueing discipline}, J. Assoc. Comput.  Mach. \textbf{33} (1986), no.~3, 593--602. \MR{0849031}

\end{thebibliography}

\providecommand{\bysame}{\leavevmode\hbox to3em{\hrulefill}\thinspace}
\providecommand{\MR}{\relax\ifhmode\unskip\space\fi MR }
\providecommand{\MRhref}[2]{%
  \href{http://www.ams.org/mathscinet-getitem?mr=#1}{#2}
}
\providecommand{\href}[2]{#2}


\ACKNO{This work has benefited from the active  participation of Ghurumuruhan Ganesan, notably on Section 3.  The author thanks also Elie A\"{i}d\'ekon and Jeremy Quastel for enlightening discussions on the analogy between the Brunet-Derrida model and the chase-escape process. 
}


\end{document}